\documentclass{conm-p-l}

\usepackage{graphicx}
\usepackage{tikz}
\usepackage{tikz-cd}
\usetikzlibrary{decorations.pathreplacing}
\usetikzlibrary{arrows}
\usepackage{multirow}
\usepackage{array}
\newcolumntype{L}{>{$\displaystyle} l <{$}}

\usepackage{calrsfs}
\usepackage{amsfonts}
\usepackage{amssymb}
\usepackage{hyperref}

\usepackage{mathtools}
\usepackage[shortlabels]{enumitem}

\newtheorem{introtheorem}{Theorem}

\theoremstyle{definition}
\newtheorem*{introdefinition*}{Definition}
\newtheorem*{introexample*}{Example}
\newtheorem*{introremark*}{Remark}
\theoremstyle{plain}
\newtheorem{theorem}{Theorem}[section]
\newtheorem{proposition}[theorem]{Proposition}
\newtheorem{lemma}[theorem]{Lemma}
\newtheorem{corollary}[theorem]{Corollary}
\newtheorem{hyp}{Hypothesis}
\newenvironment{hypothesis}[1]
  {\renewcommand\thehyp{#1}\hyp}
  {\endhyp}
\theoremstyle{definition}
\newtheorem{example}[theorem]{Example}
\newtheorem{remark}[theorem]{Remark}
\newtheorem{notation}[theorem]{Notation}

\newcommand{\tm}{\times}
\newcommand{\mc}{\mathcal}
\newcommand{\G}{\Gamma}
\newcommand{\g}{\gamma}
\newcommand{\Gm}{\mathbf{G}_m}
\newcommand{\Ga}{\mathbf{G}_a}
\DeclareMathOperator{\Z}{\mathbf{Z}}
\DeclareMathOperator{\N}{\mathbf{N}}
\DeclareMathOperator{\Q}{\mathbf{Q}}
\DeclareMathOperator{\R}{\mathbf{R}}
\DeclareMathOperator{\RR}{\mathcal{R}}
\DeclareMathOperator{\C}{\mathbf{C}}
\DeclareMathOperator{\F}{\mathbf{F}}
\DeclareMathOperator{\PP}{\mathbf{P}}
\DeclareMathOperator{\End}{\mathrm{End}}
\DeclareMathOperator{\Aut}{Aut}

\DeclareMathOperator{\Char}{char}
\DeclareMathOperator{\sgn}{sgn}
\DeclareMathOperator{\Orb}{Orb}
\DeclareMathOperator{\im}{im}
\def\ZP{\zeta^*_{\mathrm{pt}}}

\begin{document}

\title{Dynamically affine maps in positive characteristic}
\author[J.~Byszewski]{Jakub Byszewski}
\address{\normalfont (JB) Wydzia\l{} Matematyki i Informatyki Uniwersytetu Jagiello\'nskiego, ul.\ S.\ \L ojasiewicza 6,
30-348 Krak\'ow, Polska}
\email{jakub.byszewski@uj.edu.pl}
\author[G.~Cornelissen]{Gunther Cornelissen}
\address{\normalfont (GC, MH, LvdM) Mathematisch Instituut, Universiteit Utrecht, Postbus 80.010, 3508 TA Utrecht, Nederland}
\email{g.cornelissen@uu.nl,  loisvandermeijden@gmail.com}
\author[M.~Houben]{Marc Houben} 
\address {\normalfont (MH) \emph{Current address:} Departement Wiskunde, KU Leuven, Celestijnenlaan 200b---box 2400, 3001 Leuven, Belgi\"e}
\email{marc.houben@kuleuven.be}
\thanks{JB gratefully acknowledges the support of National Science Center, Poland under grant no.\ 2016/\-23/\-D/ST1/\-01124. We thank Martin Bright, Stefano Marseglia and S{\l}awomir Rams for some illuminating discussions.}


\maketitle

\begin{center}
\vspace*{-8mm} {\small {with an appendix by the authors and Lois van der Meijden}}
\end{center}

\begin{abstract} 
\noindent We study fixed points of iterates of dynamically affine maps (a generalisation of Latt\`es maps) over algebraically closed fields of positive characteristic $p$. We present and study certain hypotheses that imply a dichotomy for the Artin--Mazur zeta function of the dynamical system: it is either rational or non-holonomic,  depending on specific characteristics of the map. We also study the algebraicity of the so-called tame zeta function, the generating function for periodic points of order coprime to $p$. We then verify these hypotheses for dynamically affine maps on the projective line, generalising previous work of Bridy, and, in arbitrary dimension, for maps on Kummer varieties arising from multiplication by integers on abelian varieties. 
\end{abstract}

\section{Introduction} 

We consider so-called \emph{dynamically affine maps}, a concept in algebraic dynamics introduced by Silverman \cite[\S 6.8]{Silverman} in order to unify various interesting examples, such as Chebyshev and Latt\`es maps, cousins of which occur in complex dynamics under the name of ``finite quotients of affine maps'' or ``rational maps with flat orbifold metric'' \cite{MilnorBodilpaper}. 
We will only consider the case of a ground field of positive characteristic $p > 0$.  (Most of our methods would simplify considerably in characteristic zero and lead to results of a rather different flavour.) Before we present the definition, we will illustrate by approximative pictures (constructed in \textsc{Mathematica} \cite{Mathematica}, using the function {\tt RandomInteger} for randomisation) what distinguishes the dynamics of such maps from that of other polynomial maps and random maps.

\subsection{A compilation of (restrictions of) maps} Let $f \colon S \rightarrow S$ denote a map from a finite set $S$ to itself. It can be represented by a directed graph $D_f$ (sometimes called the ``function digraph'' of $f$), with vertices labelled by elements of $S$ and an arrow from a vertex $x$ to a vertex $y$ occurring precisely if $f(x)=y$. In Figure~\ref{random}, we plotted the graphs corresponding to two random such maps where $S$ is a set with $7^3+1$ elements.

Now consider a rational function $f \colon \PP^1(\overline \F_p) \rightarrow \PP^1(\overline \F_p)$ defined over $\F_p$ (in this subsection we assume for convenience that $p \neq 2$). To represent $f$ pictorially, consider the restrictions ${f|_{{\F_{p^N}}}} \colon \PP^1({\F_{p^N}}\!) \rightarrow \PP^1({\F_{p^N}}\!)$ for various $N$.  
In Figure~\ref{x21}, we plotted the graph of the polynomial function $x \mapsto x^2+1$ for various $p$ and $N$, and in Figure~\ref{x22}, we did the same for $x \mapsto x^2-2$. At first sight, the graph for a random map looks similar to the graph for $x \mapsto x^2+1$, but the graph for $x \mapsto x^2-2$ looks much more structured. This is no coincidence; Figure~\ref{x22} represents the graph of restrictions of a dynamically affine map, whereas Figure~\ref{x21} does not.

\begin{figure}[ht]
\begin{center}
\fbox{\includegraphics[height=60.5mm]{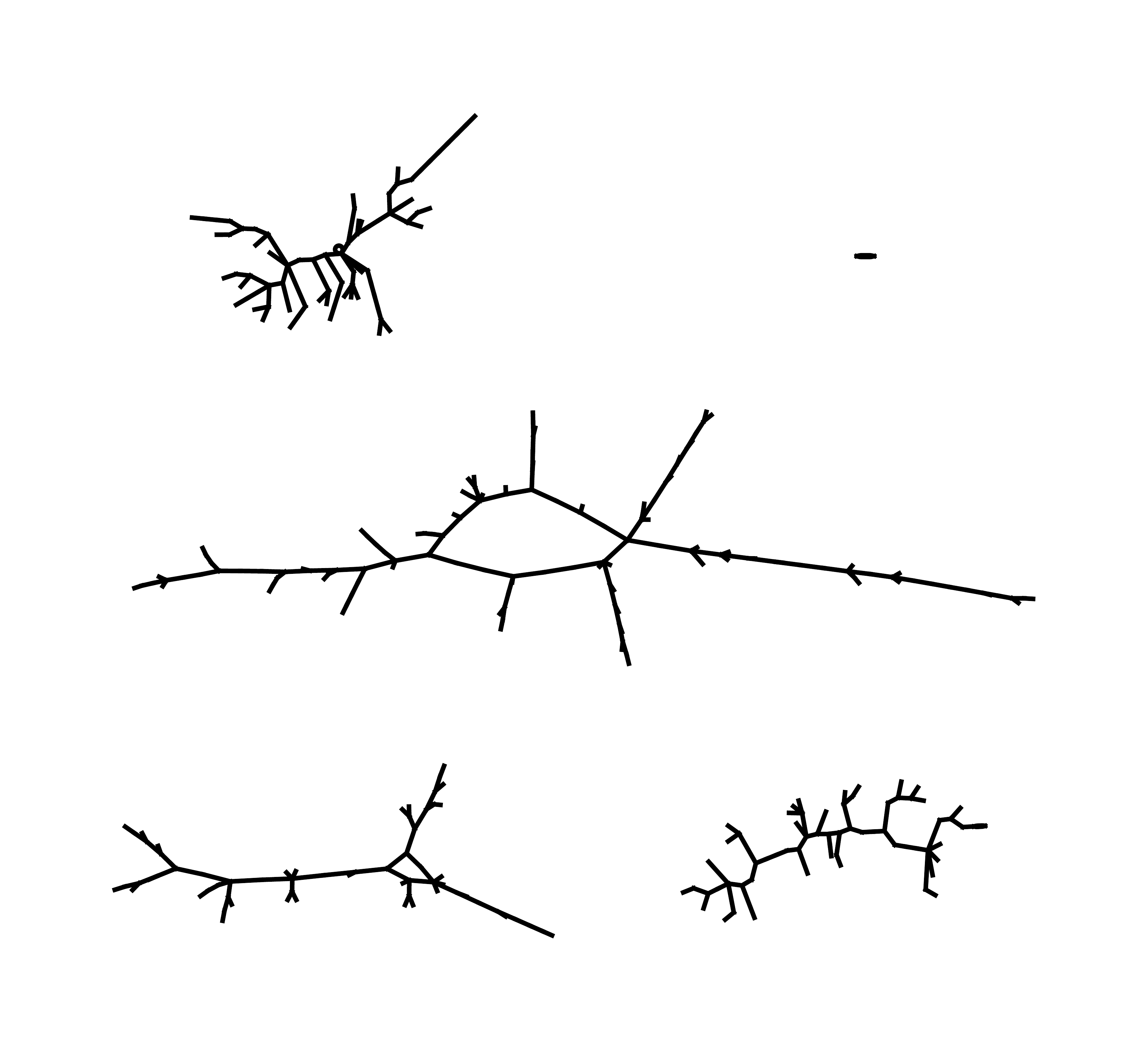}}
\fbox{\includegraphics[height=60.5mm]{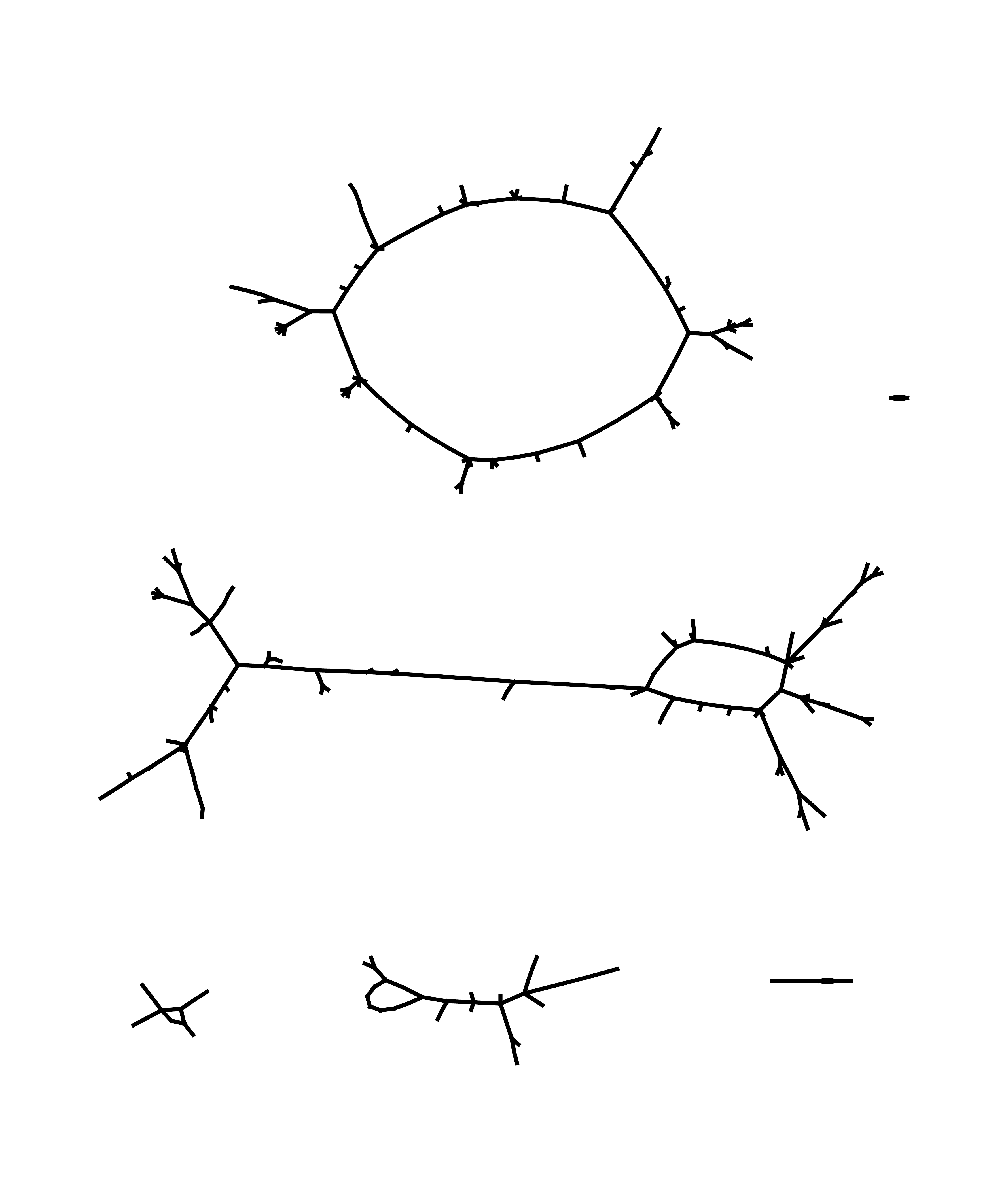}}
\caption{Graph of two random maps on a set with $7^3+1$ elements}
\label{random}
\end{center}
\end{figure}

\begin{figure}[ht]
\begin{center}
\fbox{\includegraphics[height=60mm]{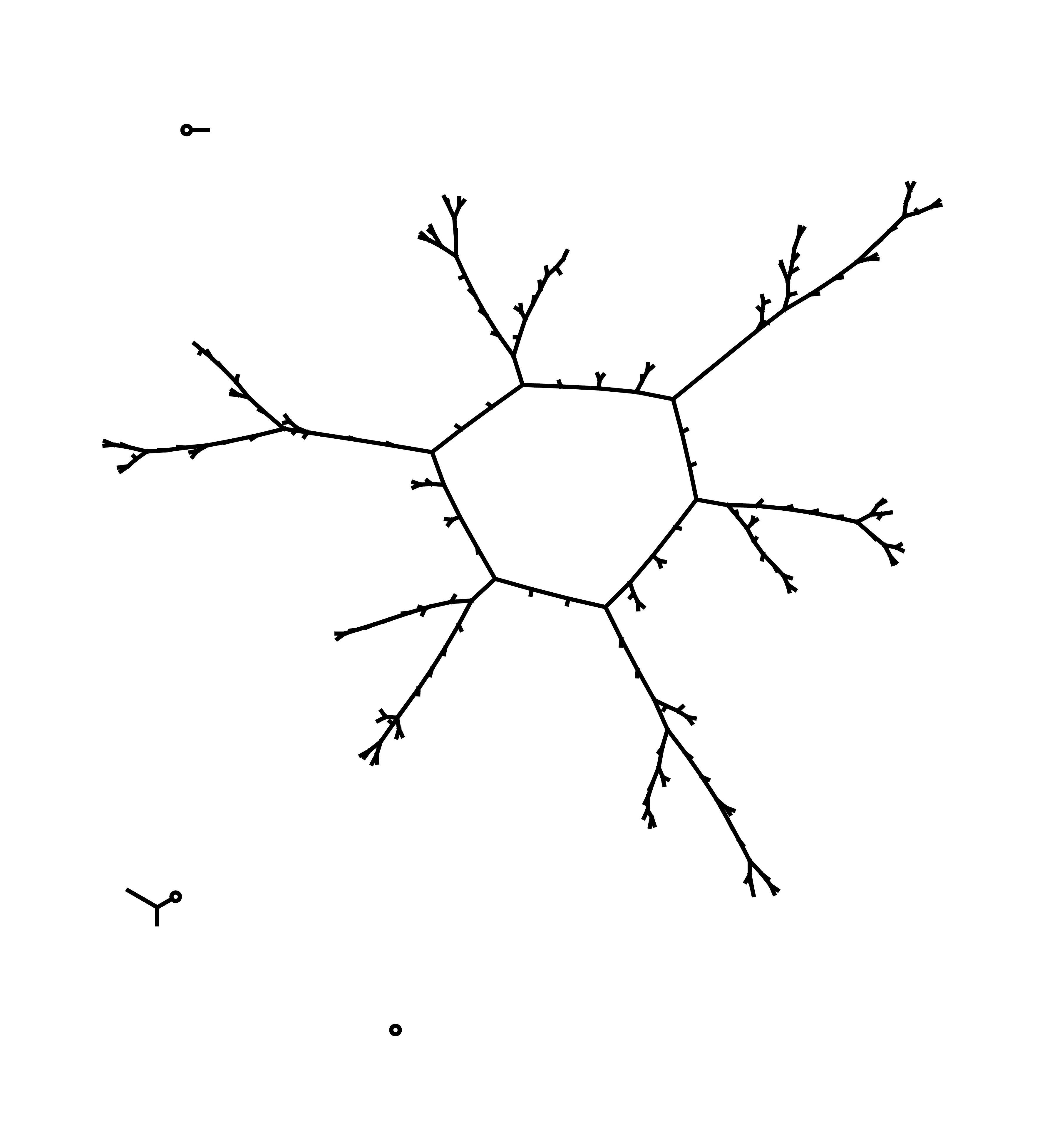}}
\fbox{\includegraphics[height=60mm]{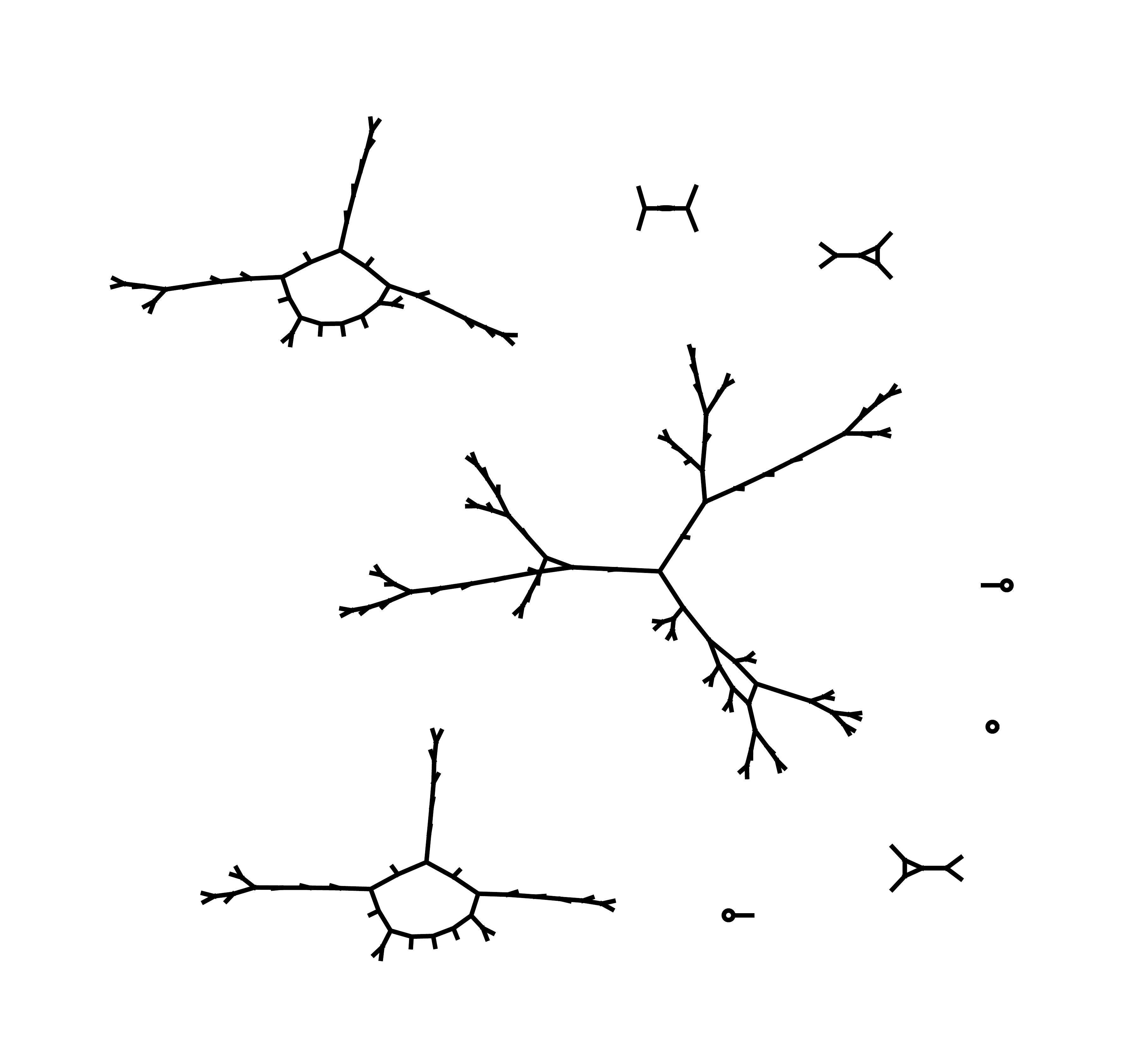}}
\caption{Graphs of $x \mapsto x^2+1$ on $\PP^1$ over a field with $7^3$ 
and $17^2$ elements (left to right)}
\label{x21}
\end{center}
\end{figure}

A common feature of all function digraphs is that their connected components are cycles (consisting of periodic points) with attached trees. What is different in Figure~\ref{x22} is the symmetry in the attached trees; this is well-understood for the polynomial $x \mapsto x^2-2$, which relates to the Lucas--Lehmer test and  failure of the Pollard rho method of factorisation, see, e.g.\ \cite{VasigaShallit, Qureshi2019}. Let us mention one further result \cite[Thm.\ 1.5 \& Example 7.2]{Kurlberg}: for the graph of a quadratic polynomial with integer coefficients, the value of $$\liminf_{p \rightarrow + \infty} \#\{ x \in \F_p \mbox{ belongs to a cycle of }D_{f\, \mathrm{mod}\, p}\}/p$$ is $0$ for $x^2+1$ but $1/4$ for $x^2-2$.

\begin{figure}[ht]
\begin{center}
\fbox{\includegraphics[height=50mm]{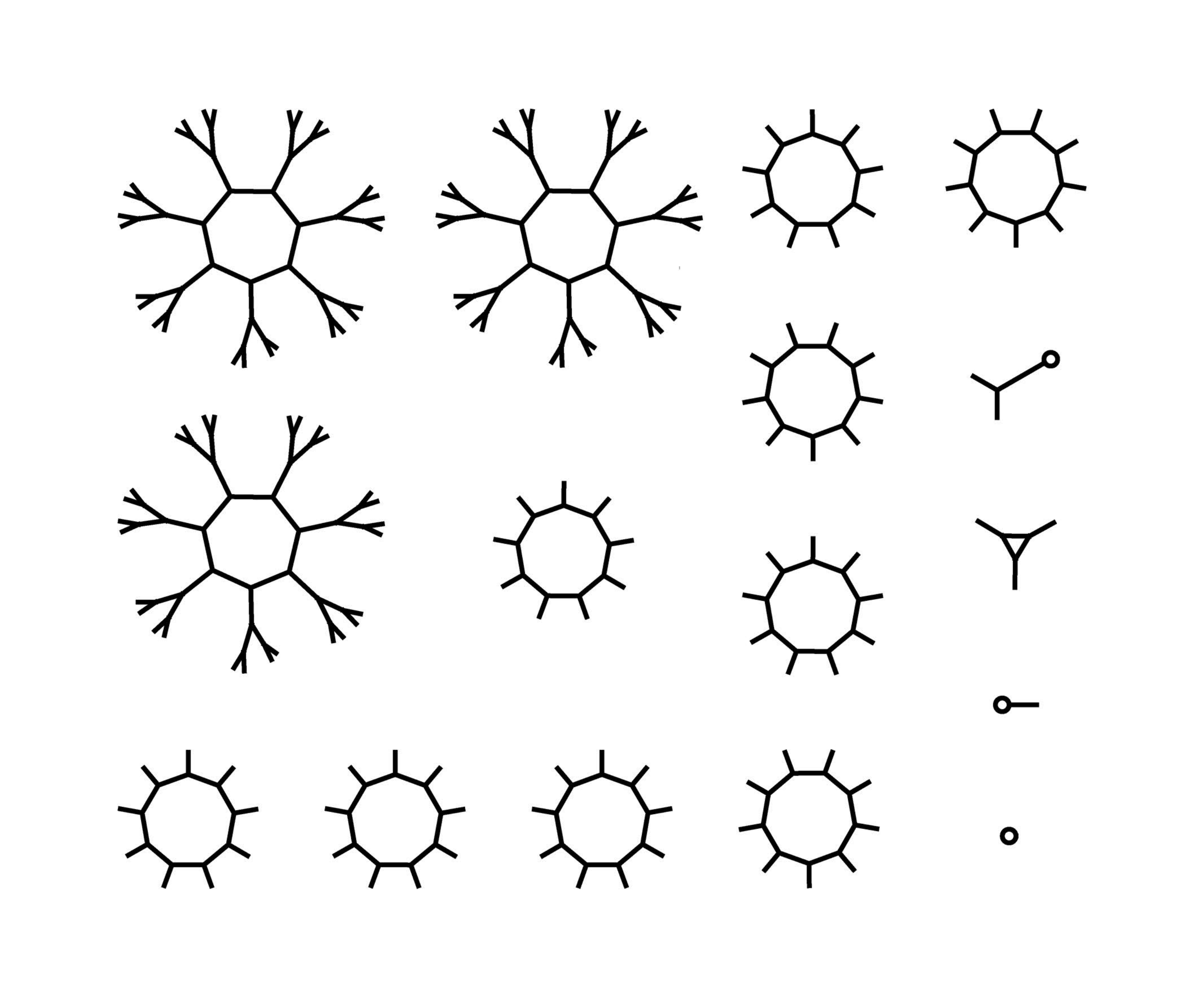}}
\fbox{\includegraphics[height=50mm]{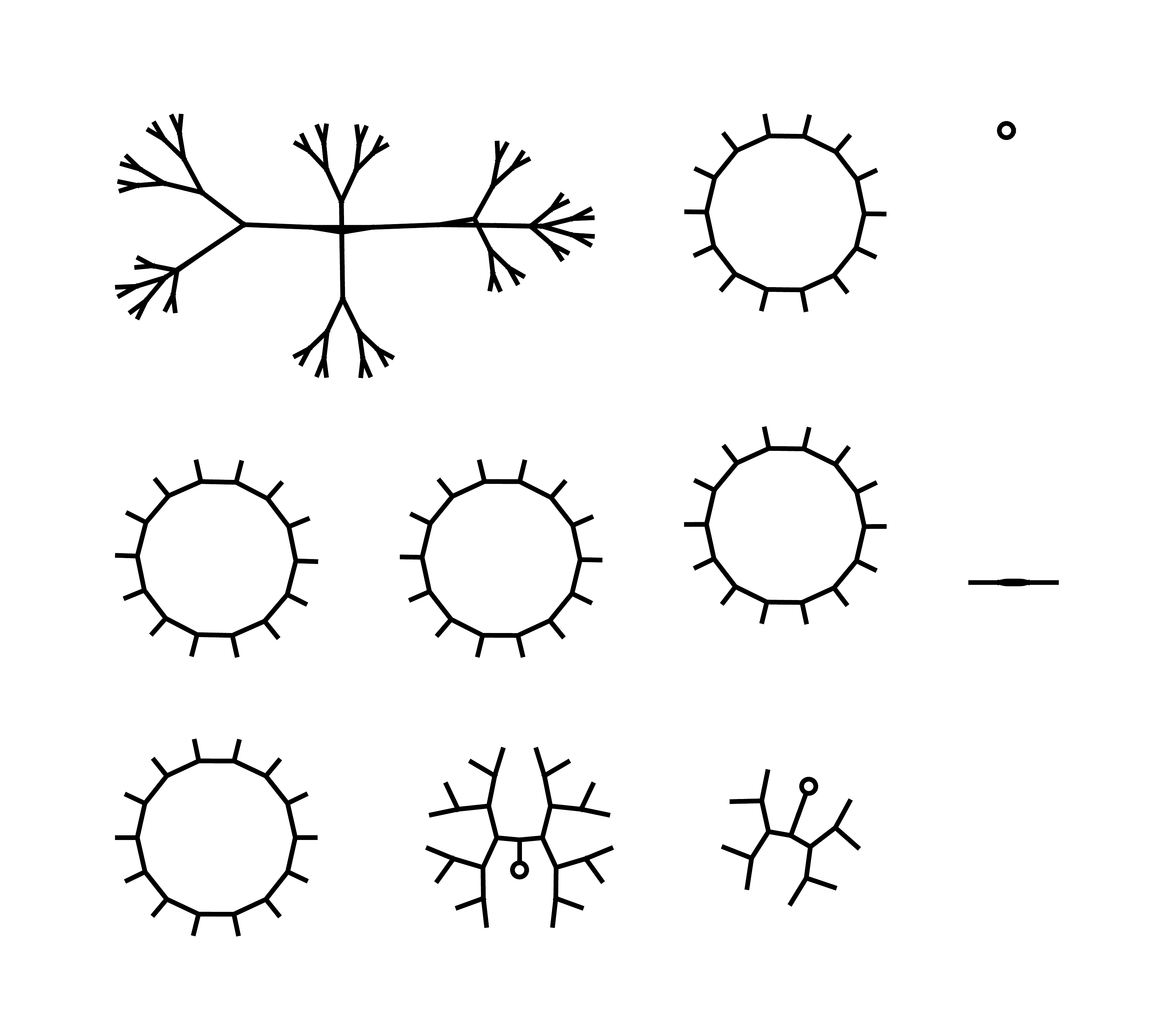}}
\caption{Graphs of $x \mapsto x^2-2$ on $\PP^1$ over a field with $7^3$  
 and $17^2$ elements (left to right)}
\label{x22}
\end{center}
\end{figure}

\begin{figure}[ht]
\begin{center}
\fbox{\includegraphics[height=52.5mm]{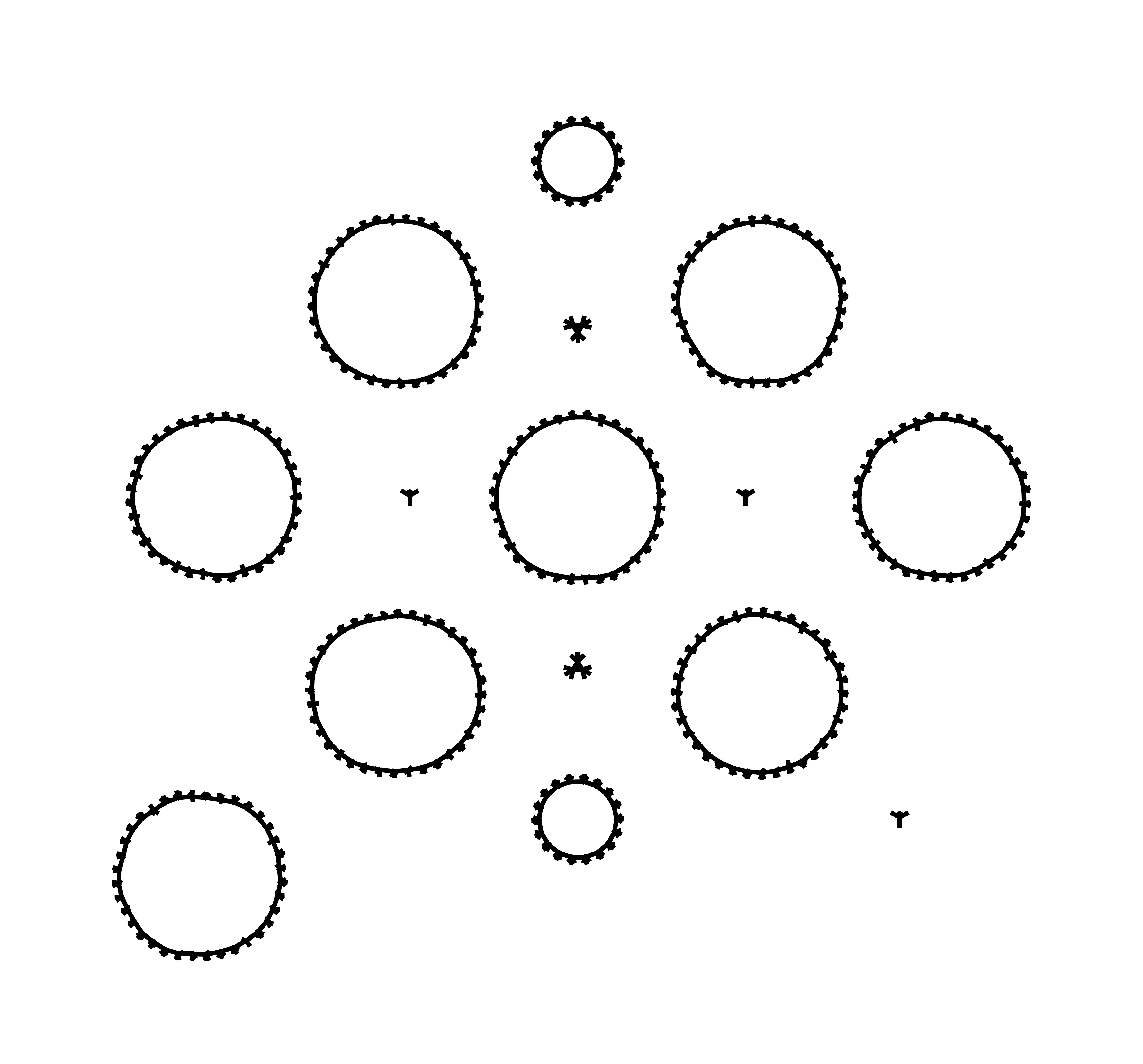}}
\fbox{\includegraphics[height=52.5mm]{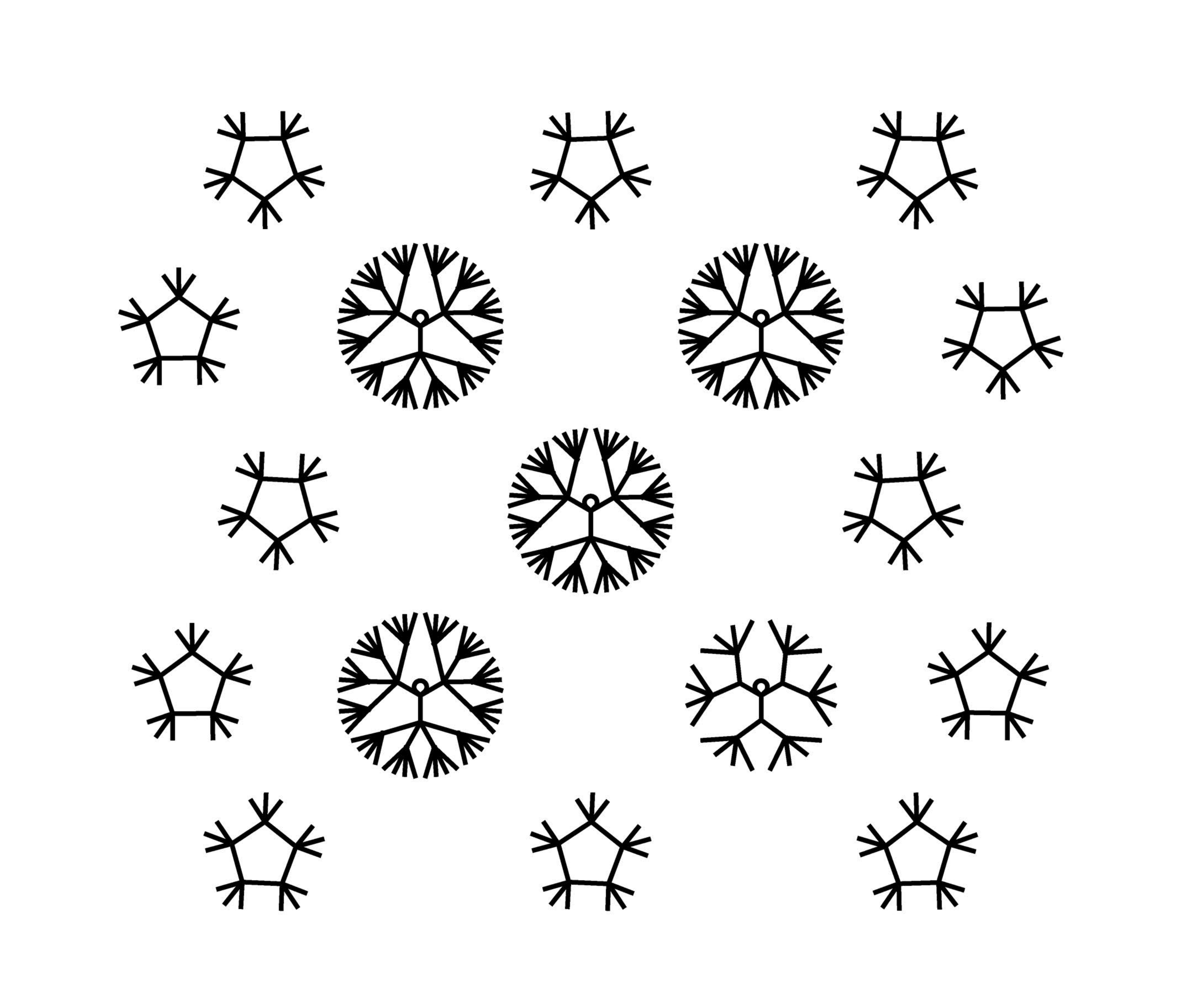}}
\caption{Graphs of the Latt\`es map arising from doubling modulo inversion on the elliptic curve $E \colon y^2=x(x-1)(x-2)$ over a field with 
$11^3$ and $23^2$ elements (left to right)} 
\label{lat}
\end{center}
\end{figure}

To explain what is special about the dynamically affine map $x \mapsto x^2-2$ as opposed to the polynomial map $x \mapsto x^2+1$, notice that $x^2-2=T_2(x)$, where $T_2$ is the normalised Chebyshev polynomial of the second kind, defined by $T_n(x+x^{-1}) =x^n+x^{-n}$. 
 This reveals a hidden group structure: the map arises from the group endomorphism $\sigma \colon \Gm \rightarrow  \Gm, x \mapsto x^2$  on the multiplicative group $\Gm$ after quotienting on both sides by the automorphism group $\Gamma=\{1,\gamma\}$  generated by the inversion $\gamma\colon \Gm \to \Gm, z \mapsto z^{-1}$ that commutes with $\sigma$. That $x^2+1$ (for $p \neq 2,3$) is not special in this sense follows from the classification of dynamically affine maps on $\PP^1$ \cite{BridyBordeaux}.  
 
We perform a similar construction using another algebraic group, the elliptic curve $E \colon y^2=x(x-1)(x-2)$, and the doubling map
$\sigma \colon E \rightarrow E, P \mapsto 2P$. After taking the quotient by $\Gamma=\langle P \mapsto -P \rangle$, we find a so-called \emph{Latt\`es map} $\PP^1 \rightarrow \PP^1$ which we have graphed over various finite fields in Figure~\ref{lat}. Again, we see a very structured picture, rather different from Figure~\ref{random} and Figure~\ref{x21}. 

We will not dwell any longer on the study of iterations of maps on finite sets, both random and ``polynomial over finite fields''---a rich subject in itself---but rather switch to our main object of study: dynamically affine maps over algebraically closed fields of positive characteristic. 

\subsection{What is a dynamically affine map?}

Let $V$ be an algebraic variety over an algebraically closed field $K$ of characteristic $p$ and $f \colon V\to V$ a morphism. 
We make the following assumption throughout: 
\begin{enumerate}
\item[\textup{\textbf{(C)}}] The map $f$ is \emph{confined}, i.e.\ the number of fixed points $f_n$ of the $n$-th iterate $f^n$ of $f$ is finite for all $n$.  
\end{enumerate}

\begin{introdefinition*}\label{def:dam} A morphism $f \colon V\to V$ of an algebraic variety $V$ over $K$ is called \emph{dynamically affine} if there exist:
\begin{enumerate}[\textup{(}i\textup{)}]
\item  a connected commutative algebraic group $(G,+)$; 
\item an \emph{affine morphism} $\psi\colon G\to G$, that is, a map of the form $$g \mapsto \psi(g)=\sigma(g)+h,$$ where $\sigma\in\End(G)$ is a confined isogeny (i.e.\ a surjective homomorphism with finite kernel) and $h\in G(K)$; 
\item a finite subgroup $\Gamma\subseteq\Aut(G)$; and 
\item a morphism $\iota\colon \Gamma\backslash G \to V$ that identifies $\Gamma\backslash G $ with a Zariski-dense open subset of $V$
\end{enumerate} such that the following diagram commutes:
\begin{equation}\label{eq:dam}
\begin{tikzcd}        
   G \arrow{r}{\psi} \arrow{d}{\pi} & G    \arrow{d}{\pi} \\
    \Gamma\backslash G \arrow[hook]{d}{\iota} \arrow{r}{}  &  \Gamma\backslash G \arrow[hook]{d}{\iota} \\
    V \arrow{r}{f}& V.
\end{tikzcd}
\end{equation}
\end{introdefinition*}

\begin{introremark*} In this paper, we adhere to the convention that a dynamically affine map consists of all the given (fixed) data in the definition, so that we can refer to the constituents $(G,\psi, \sigma, h, \Gamma, \iota)$ directly. The same map $f$ might arise from different sets of data, and in our sense be a different dynamically affine map despite being the same map on $V$. 
\end{introremark*}

\begin{introexample*}
As explained above, the map $\PP^1 \rightarrow \PP^1, x\mapsto x^2-2$ is dynamically affine for the data $(G=\Gm, \sigma \colon x \mapsto x^2, h=1, \Gamma =\langle z \mapsto z^{-1} \rangle, \iota \colon \Gamma\backslash \Gm \simeq \mathbf{A}^1 \hookrightarrow \PP^1)$ (written multiplicatively); its restrictions (to certain finite fields) were represented in Figure~\ref{x22}.

The map $\PP^1 \rightarrow \PP^1, x \mapsto (x^4-4x^2-4)/(4x(x-1)(x-2))$ is dynamically affine for the data $(G=E, \sigma \colon P \mapsto 2P, h=0_E, \Gamma =\langle P \mapsto -P \rangle, \iota \colon \Gamma\backslash E \cong \PP^1 \rightarrow \PP^1)$, where $E$ is the elliptic curve $y^2=x(x-1)(x-2)$; its restrictions were represented in Figure~\ref{lat}. 
\end{introexample*} 

\begin{introremark*}
We have slightly modified Silverman's definition \cite[\S 6.8]{Silverman} of a dynamically affine map. Instead of assuming confinedness of $\sigma$, Silverman imposes the condition $\deg(\sigma)\geqslant 2$ (as in Er\"emenko's classification theorem \cite{Eremenko}). As long as $G$ is one-dimensional and $K=\overline{\F}_p$, the definitions are equivalent. 

In a general setup one could assume merely that $\sigma$ is an isogeny and only require $f$ to be confined. This reduces, after  some case distinctions,  to the case where $\sigma$ is a confined isogeny, so we choose to put the latter property in the definition. \end{introremark*}

\subsection{Counting fixed points, orbits, and the dynamical Artin--Mazur zeta function} A natural way to begin a quantitative analysis of a discrete dynamical system such as iteration of a map $f \colon V \rightarrow V$ is to consider the sequence $(f_n)$ given by the number of fixed points of the $n$-th interate of $f$. Confinedness implies that this is a well-defined sequence of integers, and we can form 
the (full) Artin--Mazur dynamical zeta function (\cite{AM}, \cite[\S 4]{Smale}) defined as
\begin{equation} \zeta_{f}(z) := \exp \left( \sum_{n \geqslant 1} f_n \frac{z^n}{n} \right). \end{equation} 
We consider this a priori as a formal power series, but the question of convergence in a neighbourhood of $z=0$ (equivalent to $f_n$ growing at most exponentially in $n$) is interesting, and we study this in Appendix \ref{conv}. 

Counting fixed points and closed orbits is related: if $P_\ell$ denotes the number of closed orbits of length $\ell$, then $f_n = \sum_{\ell{\mid}n} \ell P_\ell,$ and there is an ``Euler product'' \begin{equation} \label{euler} \zeta_f(z) = \prod_C \frac{1}{1-z^{\ell(C)}}, \end{equation} where the product runs over the closed orbits $C$.  

It is interesting to understand the nature of the function $\zeta_f(z)$ (Smale \cite[Problem 4.5]{Smale}); Artin and Mazur \cite[Question 2 on p.~84]{AM}). For example, rationality or algebraicity of $\zeta_f(z)$ means that there is an easy recipe to compute all $f_n$ from a finite amount of data (in the rational case, it implies that $(f_n)$ is linearly recurrent). Zeta functions of more general dynamical systems can:
\begin{description}[before={\renewcommand\makelabel[1]{ ##1:}}, labelindent=0.4cm]

\item[$-$ \emph{be rational}] e.g.\ for ``Axiom A'' diffeomorphisms by Manning \cite[Cor.~2]{Manning}, for rational functions of degree $ \geqslant 2$ on the Riemann sphere by Hinkkanen  \cite[Thm.~1]{Hinkkanen}, for the Weil zeta function (when $f$ is the Frobenius map on a variety defined over a finite field) by Dwork \cite{Dwork} and Grothendieck  \cite[Cor.~5.2]{Gr}, for endomorphisms of real tori \cite[Thm.~1]{Baake}, and when $f_n$ replaced by the Lefschetz number of $f^n$ \cite{Smale};
\item[$-$ \emph{be algebraic but not rational}] e.g.\ when $f$ is an orientation preserving surface homeomorphism and $f_n$ is replaced by the Nielsen number of $f^n$ by  Pilyugina and Fel'shtyn \cite{PF}, \cite[Thm.~36]{F}; 
\item[$-$ \emph{be transcendental}] e.g.\ for restrictions of shifts by Bowen and Lanford \cite[\S 3--4]{BowenLanford} and for separable dynamically affine maps on $\PP^1(\overline \F_p)$ by Bridy \cite[Thm.\ 1]{Bridy}, \cite[Thm.\ 1.2 \& 1.3]{BridyBordeaux};
\item[$-$ \emph{have an essential singularity}] e.g.\ for some flows by Gallavotti \cite[\S 4]{Gallavotti}; 
\item[$-$ \emph{have a natural boundary}] e.g.\ for certain beta-transformations by Flatto, Lagarias and Poonen \cite[Thm.~2.4]{FLP}, for some $\Z^d$-actions ($d \geqslant 2$) by Lind \cite{Lind}, for some flows by Pollicott  \cite[\S 4]{Pollicott} and Ruelle \cite{Ruelle}, for a ``random'' such zeta function by Buzzi \cite{Buzzi}, for some explicit automorphisms of solenoids by Bell, Miles, and Ward \cite{BMW}, and for most endomorphisms of abelian varieties in characteristic $p>0$ by the first two authors \cite{D1}.
\end{description}

Following the philosophy of \cite{D1}, we will also study ``tame dynamics'' via the so-called \emph{tame zeta function} defined by 
\begin{equation} \label{deftame} \zeta^*_{f}(z) := \exp \left( \sum_{p \nmid n} f_n \frac{z^n}{n} \right), \end{equation} 
summing only over $n$ that are not divisible by $p$. Tame and ``full'' dynamics are related by the formulae in \eqref{eq:tameexpand} below, but the tame zeta function tends to be better behaved. In Appendix \ref{lois}, we give some explicit expressions for the tame zeta function of several dynamically affine maps on $\PP^1$. 

\subsection{Main results} 
Bridy studied the zeta function for dynamically affine maps on $V=\PP^1$. The main results in \cite[Thm.\ 1.2 \& 1.3]{BridyBordeaux} imply that if $f$ is dynamically affine for $V=\PP^1$ and $K=\overline \F_p$, then $\zeta_f(z)$ is transcendental over $\C(z)$ if and only if $f$ is separable; otherwise $\zeta_f(z)$ is rational. Bridy's full result applies to all $K$; the proof uses a case-by-case analysis (see Table \ref{tab:da} in Appendix \ref{lois} below) and is based on the relation between transcendence and automata theory. This starkly contrasts with the fact that in characteristic zero all dynamically affine maps have a rational zeta function (a much more general result by Hinkkanen was quoted above). 

In this paper, we prove a strengthening of Bridy's result. For this, we need some further concepts. Let $f \colon V\to V$ be a dynamically affine map.

\begin{introdefinition*}
An endomorphism $\sigma\in\End(G)$ is said to be \emph{coseparable} if $\sigma^n-1$ is a separable isogeny for all $n\in\Z_{>0}$. A dynamically affine map $f$ is called coseparable if the associated isogeny $\sigma$ is coseparable.
\end{introdefinition*}

\begin{introremark*}
In \cite{D1}, we called a coseparable endomorphism of an abelian variety ``very inseparable'' and showed that this implies inseparability  \cite[6.5(ii)]{D1}. However, it is not true that coseparable dynamically affine maps are inseparable in general. For example, if $f$ is the map   $f \colon \PP^1 \rightarrow \PP^1, x \mapsto tx$ for $t \in K$ transcendental over $\overline \F_p$, then $f$ is both coseparable and separable (a more general statement is given in \cite[Thm.~1.3]{BridyBordeaux}). \end{introremark*}

 \begin{introdefinition*} A holomorphic function on a connected open subset $\Omega \subseteq \C$ is said to have a \emph{natural boundary} along $\partial \Omega$ if it has no holomorphic continuation to any larger such $\Omega' \varsupsetneq \Omega$ \cite[\S 6]{SS}. 
We call a function $F(z)$ \emph{root-rational} if $F(z)^t \in \C(z)$ for some $t \in \Z_{>0}$. 
We call $F(z)$ \emph{holonomic} if it satisfies a nontrivial linear differential equation with coefficients in $\C(z)$. 
\end{introdefinition*} 
\noindent Since algebraic functions are holonomic \cite[Thm.\ 6.4.6]{StanleyEC2}, the following is indeed a strengthening of Bridy's result. At the same time, it shows that ``tame'' dynamics is better behaved. 

\begin{introtheorem} \label{thm:main} Assume $f\colon \PP^1 \rightarrow \PP^1$ is a dynamically affine map.  
\begin{enumerate}[\textup{(}i\textup{)}]
\item\label{thm:main1} If $f$ is coseparable, $\zeta_f(z)$ is a rational function; otherwise, $\zeta_f(z)$ is not holonomic; more precisely, it is a product of a root-rational function and a function admitting a natural boundary along its circle of convergence.
\item\label{thm:main2}  For all $f$, $\zeta_f^*(z)$ is root-rational; equivalently, it is algebraic and satisfies a first order differential equation over $\C(z)$. 
\end{enumerate}
\end{introtheorem}

\noindent We mention an amusing corollary of Theorem \ref{thm:main}: although $\zeta_f(z)$ is in general not holonomic, the \emph{pair} $(\zeta_f(z),\zeta_{f^p}(z))$ always satisfies a simple differential equation; see  Corollary \ref{deq}  for a precise statement. 

Rather than using results from automata theory, we prove Theorem \ref{thm:main} essentially relying on a method of Mahler (see  \cite{BCR}).  We structure the proof abstractly, showing the result for dynamically affine maps (in any dimension) that satisfy certain hypotheses \textup{\textbf{(H1)}--\textbf{(H4)}} (see  Section \ref{sec:generalH}), and then verify these for $V=\PP^1$.  

We give a more general discussion of when the hypotheses hold or fail, in this way producing the first higher-dimensional examples of dynamically affine maps in positive characteristic with nontrivial $\Gamma$ where we understand the nature of the dynamical zeta function. Recall that the quotient of an abelian variety $A$ by the group $\Gamma = \{[\pm 1]\}$ is called a \emph{Kummer variety}. 

\begin{introtheorem} \label{higherdim} Let $V$ denote a Kummer variety arising from an abelian variety $A$, and let $f \colon V \rightarrow V$ denote the dynamically affine map induced by the multiplication-by-$m$ map $\sigma=[m]$ for some integer $m \geqslant 2$. Then $\zeta_f^*(z)$ is root-rational. The function $\zeta_f(z)$ is not holonomic if $m$ is coprime to $p$ and rational otherwise. 
\end{introtheorem} 

\begin{introremark*}
We use the word ``Kummer variety'' for the variety $V=\Gamma\backslash A$ that, for $\dim A>1$, is singular at points in the finite subset $\Gamma\backslash A[2]$ of $V$, but the name is sometimes used for the minimal resolution $\tilde V$ of $V$. Since the set of singular points is finite and stable by $f$, the map $f$ can be seen as a birational map $f\colon \tilde V \dashrightarrow \tilde V$ with locus of indeterminacy stable by $f$, and the above theorem can be interpreted as a statement about the periodic points of this birational map outside the preimage of the singular points.  
\end{introremark*}

\begin{introremark*}
The non-holonomicity shows that the sequence $(f_n)$ of number of fixed points of the iterates of $f$ is somewhat ``complex'', but it does not mean that $f_n$ is ``uncomputable''. As a matter of fact, the results in \cite{BCH} say that for $f$ an endomorphism of an algebraic group there exists a formula expressing $f_n$ in terms of a linear recurrent sequence and two specific periodic sequences of integers that control a $p$-adic deviation of $f_n$ from being linearly recurrent. These data can in principle be computed by breaking up the algebraic group into abelian varieties, tori, vector groups, and semisimple groups. Similarly, one can in principle trace through our proofs to compute $f_n$ for dynamically affine maps satisfying our hypotheses. 
\end{introremark*}

\noindent We finish the introduction by mentioning a few possibilities for future research. 
\begin{itemize}[leftmargin=*, labelindent=0.4cm]
\item[$-$] The relation between fixed points and closed orbits may be used to study the distribution of closed orbit lengths (analogously to the prime number theorem). Because of the analytic nature of the function $\zeta_f(z)$ revealed by our results, one cannot in general use standard Tauberian methods. We have studied this question via a different route for maps on abelian varieties \cite{D1} and for maps on general algebraic groups \cite{BCH} (which covers the case of dynamically affine maps with trivial $\Gamma$, $h$, and $\iota$, but is more general, since we do not require the group $G$ to be commutative). It would be interesting to extend this to general dynamically affine maps. 
\item[$-$] We have no good understanding of the dynamical zeta function of general rational functions on $\PP^1$ that are not dynamically affine, e.g.\ $x \mapsto x^2+1$ in characteristic $p\geqslant 5$ (see  \cite[Question 2]{Bridy}). It would be interesting to investigate the nature of the (tame) zeta function for such examples. 
\item[$-$] Inhowfar the hypotheses  \textup{\textbf{(H1)}--\textbf{(H4)}} are necessary to reach the conclusion of the main theorem merits attention, since they are extracted from a ``method of proof'' rather than intrinsic. 
\item[$-$] In general, $V$ may be singular. It is interesting to study whether $V$ admits a resolution to which $f$ extends as a morphism, and the relation between the zeta function of that extended morphism and the zeta function of $f$. This is nontrivial already for Kummer surfaces (where, for $p>2$, the minimal resolution is a K3 surface, and hence has trivial \'etale fundamental group \cite[pp.\ 3--6]{HuybrechtsK3}). 
\end{itemize} 
\noindent The structure of the paper is as follows: After some generalities, we introduce the hypotheses in Section \ref{sec:generalH} and prove the main result, conditional on the hypotheses, in the following section. Then, in Section \ref{sec:disc} we discuss the validity of the hypotheses in various settings (giving examples and counterexamples). The main theorems then follow immediately from these results. In the first appendix, we consider the radius of convergence of $\zeta_f(z)$, and in the second appendix, we compute a collection of examples of tame zeta functions of dynamically affine maps. 
  
\section{Generalities} \label{generalities}

\subsection*{Relations between zeta functions} 
\begin{proposition}\label{prop:zetaexpands}
The tame and full dynamical zeta function are related by the following equalities of formal power series:
\begin{equation}\label{eq:tameexpand}
\zeta_f^*(z)=\frac{\zeta_f(z)}{\sqrt[p]{\zeta_{f^{p}}(z^p)}},\qquad \zeta_f(z)=\prod_{i\geqslant 0}\sqrt[p^i]{\zeta_{f^{p^i}}^*(z^{p^i})}.
\end{equation}
\end{proposition}
\begin{proof}
For the first equality, note that
\begin{equation*}
\log\zeta_f^*(z)=\sum_{n\geqslant 1}f_n \frac{z^n}{n}-\frac{1}{p}\sum_{m\geqslant 1}f_{pm}\frac{z^{pm}}{m}=\log\left(\zeta_f(z)\zeta_{f^{p}}(z^p)^{-1/p}\right).
\end{equation*}
The second equality follows by applying the first one to the functions $f^{p^i}$ for $i\in \Z_{\geqslant 0}$.
\end{proof}

\begin{remark}\label{rem:zetaMV}
A useful computational fact is the following: if $f \colon S \rightarrow S$ is a map and $S$ decomposes as a union $S=S_1 \cup S_2$ with $f(S_1) \subseteq S_1$ and $f(S_2) \subseteq S_2$, then 
$$ \zeta_f(z) = \frac{\zeta_{f|_{S_1}}(z) \cdot \zeta_{f|_{S_2}}(z)}{ \zeta_{f|_{S_1 \cap S_2}}(z)}, $$
and similarly for $\zeta_f^*(z)$. 
\end{remark}

\subsection*{Recurrences} We recall some well-known facts (see  e.g.\ \cite[\S 1]{D1}).
If $(a_n)_{n\geqslant 1}$ is a sequence of complex numbers, then the ordinary generating function 
$\sum_{n\geqslant 1}a_n z^n
$
 is rational if and only if the sequence is linear recurrent, and if and only if 
there exist $\lambda_i\in\C^{\tm}$ and polynomials $p_i\in\C[z]$ such that
\begin{equation} \label{lrec} a_n=\sum_{i=1}^r p_i(n)\lambda_i^n \end{equation} for sufficiently large $n$. The statement that the  zeta function
\begin{equation}\label{eqn:zetadfn} F(z)=\exp\left(\sum_{n\geqslant 1} a_n \frac{z^n}{n}\right)
\end{equation}
 is rational is stronger: this happens if and only if Equation (\ref{lrec}) holds for all $n\in\Z_{>0}$ with the $p_i(n)$ replaced by integers $m_i$ independent of $n$. The $\lambda_i$ occurring in (\ref{lrec}) are called the \emph{roots} of the recurrence, the polynomials $p_i$ their \emph{multiplicities}. We say that $(a_n)$ satisfies the \emph{dominant root assumption} if there is a unique root $\lambda_i$ of maximal absolute value, possibly with multiplicity $\neq 1$. 
 
 For a zeta function $F(z)$ in \eqref{eqn:zetadfn}, we may consider its tame variant $$F^*(z)=\exp\left(\sum_{\substack{n\geqslant 1\\p{\nmid}n}} a_n \frac{z^n}{n}\right).$$ It follows from the formula \begin{equation}\label{eqn:finaltame} F^*(z)=F(z)\cdot\left( \prod_{j=0}^{p-1}F(e^{\frac{2i\pi j }{p}}z)\right)^{-1/p}\end{equation} that if $F(z)$ is rational, then $F^*(z)$ is root-rational.
 
 \subsection*{Algebraicity properties and differential equations} If a formal power series $F(z)$ satisfies a nontrivial linear differential equation over $\C(z)$, it is said to be \emph{holonomic}. If $F(z)$ is algebraic over $\C(z)$, it is holonomic \cite[Thm.\ 6.4.6]{StanleyEC2}. On the other hand, if $F(z)$ converges on some nontrivial open disc $D$ and has natural boundary along $\partial D$, then it cannot be holonomic, since a holonomic function has only finitely many singularities (for a precise statement, see \cite[4(a)]{Stanley}). 

The equivalence statement in Theorem \ref{thm:main}\ref{thm:main2} is implied by the following lemma, which is certainly well-known, but for which we were unable to find a convenient reference. (A more general result can be found in \cite[Exercise 6.62]{StanleyEC2} together with an argument attributed to B.~Dwork and M.~F.~Singer.)

\begin{lemma} \label{diffeq} An algebraic function $F(z)\in \C(\!(z)\!)$ is root-rational if and only if  $f$ satisfies a first order homogeneous differential equation $F'(z)= R(z) F(z)$ with $R(z)\in \C(z)$.\end{lemma} 

\begin{proof} First assume that $F(z)$ is root-rational, i.e.\ $F(z)=q(z)^k$ with $q(z)\in \C(z)$, $k\in \Q$. We may assume that $q(z)\neq 0$, and then $F(z)$ satisfies the equation $F(z)'=R(z) F(z)$ with $R(z)=\frac{kq'(z)}{q(z)}$.

The converse direction can be proven by direct integration and partial fraction expansion of $R(z)$, but we give a somewhat different argument. Assume that $F(z)$ satisfies the equation $F(z)'=R(z)F(z)$ with $R(z)\in \C(z)$, where we may assume $R(z)\neq0$. Let $P\in \C(z)[t]$ be the minimal polynomial of $F(z)$ over $\C(z)$. Write $P=t^d+a_{d-1}(z)t^{d-1}+\dots+a_0(z)$ with $a_i(z)\in\C(z)$. Differentiating the equation $P(F(z))=0$ gives \begin{equation} \label{eqn:ratalg} P^D(F(z))+P'(F(z))F(z)'=0,\end{equation} where $P^D=\sum_{i=0}^{d} a_i'(z) t^i$ is obtained from $P$ by differentiating the coefficients and $P'=\sum_{i=0}^d ia_i(z)t^{i-1}$ is the usual derivative of $P$. Substituting $F'(z)=R(z)F(z)$ into (\ref{eqn:ratalg}), we see that $F$ is a root of the polynomial $P^D + tR(z)P'$, which is a polynomial of degree $d$ with leading coefficient $dR(z)$, and hence $$P^D+tR(z)P'=dR(z)P.$$ Comparing the coefficients at $t^i$ for $i=0,\ldots,d-1$, we see that each $a_i(z)$ satisfies the equation $$a_i'(z)=(d-i)R(z)a_i(z),$$ which differs from the equation satisfied by $F(z)$ only by a multiplicative constant. Comparing these solutions gives $a_i(z)=c_iF(z)^{d-i}$ for some $c_i\in \C$. If $a_i(z)=0$ for all $i\in\{1,\ldots,d\}$, we get $F(z)=0$. Otherwise, for some $i$ we have $a_i(z)\neq0$, and $F(z)=(c_i^{-1}a_i(z))^{1/(d-i)}$ is root-rational.
\end{proof}

Thus, Theorem  \ref{thm:main}\ref{thm:main2} immediately implies the result alluded to in the introduction: 

\begin{corollary}  \label{deq} If $f \colon {\PP^1} \rightarrow \PP^1$ is a dynamically affine map, then the \emph{pair} of zeta functions $(F_1(z),F_2(z)) = (\zeta_f(z), \zeta_{f^p}(z))$ satisfies a nonlinear first order differential equation  
$$ F_1'(z) F_2(z^p) - F_1(z) F_2'(z^p) z^{p-1} = R(z) F_1(z) F_2(z^p) $$ for some rational function $R(z) \in \C(z)$, regardless of whether of not $f$ is coseparable.  
\end{corollary}

\begin{proof}
The root-rationality of $\zeta_f^*(z)$ implies that it satisfies a differential equation of the form $(\zeta_f^*(z))'=R(z)\zeta^*_f(z)$ for some rational function $R(z) \in \C(z)$. The result follows by taking derivatives in the first identity in (\ref{eq:tameexpand}). 
\end{proof} 

\section{Introduction of the general hypotheses} \label{sec:generalH}
Let $f\colon V\to V$ be a dynamically affine map with data as in diagram \eqref{eq:dam}. Denote by $\Orb_f(x):=\{f^n(x)\mid n\in\Z_{\geqslant 0}\}$ the forward orbit of $x\in V(K)$ under $f$. For an isogeny $\tau\in\End(G)$, we denote by $\deg(\tau)$ and $\deg_{\mathrm{i}}(\tau)$ the degree and inseparable degree of the field extension $K(G)/\tau_*K(G)$, respectively. Then we have
\begin{equation}\label{eq:kerdegs}
\#\ker(\tau)=\deg(\tau)/\deg_{\mathrm{i}}(\tau).
\end{equation}
The following lemma, taken from \cite[Lemma 2.4]{BridyBordeaux} (cf.~Remark \ref{rem:Bridylemform}), will be crucial to control  the sequence  $(f_n)$, as it allows us to express $f_n$ in terms of kernels of isogenies on the algebraic group $G$. The proof will be given in Section~\ref{pra}. 
\begin{lemma}\label{lem:key}
Let $f\colon V\to V$ be a dynamically affine map. Consider the set $$C:=\{x\in V(K)\mid\Orb_f(x)\cap\iota(({\Gamma}\backslash G) (K))=\emptyset\}.$$ Then
\begin{equation}\label{eq:key}
f_n=(f|_C)_n+\frac{1}{|\Gamma|}\sum_{\gamma\in\Gamma}\#\ker(\sigma^n-\gamma).
\end{equation}
\end{lemma}
\noindent Combining Lemma \ref{lem:key} with \eqref{eq:kerdegs}, we see that in order to understand the sequence $(f_n)$ it suffices to control, for every $\gamma\in\Gamma$,
\begin{enumerate}[\textup{(}a\textup{)}]
\item \label{seq:a} the sequence $(f|_C)_n$;
\item \label{seq:b} the ``inseparable degree sequence'' $\deg_{\mathrm{i}}(\sigma^n-\gamma)$;
\item \label{seq:c} the ``degree sequence'' $\deg(\sigma^n-\gamma)$.
\end{enumerate}
Notice that the translation parameter $h\in G(K)$ no longer occurs in \eqref{eq:key}.

We now introduce the four hypotheses that we require in order to prove the main theorems. The first three hypotheses \ref{h1}, \ref{h2} and \ref{h3} are employed to control the sequences \ref{seq:a}, \ref{seq:b} and \ref{seq:c}, respectively, while \ref{h4} is a technical hypothesis that we require to avoid an unexpected cancellation of singularities in our proof of the existence of a natural boundary. 

We use the following \begin{quote} \textsc{convention}: If a hypothesis is assumed in an environment (definition, lemma, theorem, hypothesis, \dots), we label the environment by this hypothesis in square brackets. \end{quote}

\begin{hypothesis}{\textup{\textbf{(H1)}}} \label{h1} The zeta function corresponding to $f|_C$ is rational.
\end{hypothesis}

For the second hypothesis, we recall the following notion: a \emph{discrete valuation} on a (not necessarily commutative) ring $R$ is a map $v\colon R\to\Z\cup\{\infty\}$ such that for all $\tau,\tau_1,\tau_2\in R$ we have $v(\tau)=\infty$ if and only if $\tau=0$, $v(\tau_1\tau_2)=v(\tau_1)+v(\tau_2)$, and $v(\tau_1+\tau_2)\geqslant\min\{v(\tau_1),v(\tau_2)\}$. It follows from these properties that $v(\tau_1+\tau_2)=\min\{v(\tau_1),v(\tau_2)\}$ whenever $v(\tau_1)\neq v(\tau_2)$.

\begin{hypothesis}{\textup{\textbf{(H2)}}} \label{h2} Both $\sigma$ and $\Gamma$ belong to a subring $\RR$ of $\End(G)$ all of whose nonzero elements are isogenies, and such that there exists a discrete valuation $v\colon\RR\to\Z \cup\{\infty\}$ satisfying $\deg_{\mathrm{i}}(\tau)=p^{v(\tau)}$ for all isogenies $\tau\in \RR$. 
\end{hypothesis}

Note that the valuation $v$ considered in \ref{h2} takes only nonnegative values.

 Before introducing the last two hypotheses, we set up some notation. 
 \begin{notation}\label{not:sm} Let $v$ be as in \ref{h2}. For $m\in\Z_{\geqslant 0}$, we let $$\G_m:=\{\g\in\G\mid v(\g-1)\geqslant m\}.$$ This defines a descending filtration of normal subgroups of $\G$ $$\G=\G_0\supseteq\G_1\supseteq\cdots\supseteq\G_N=1,$$ where $$N:=\max\{v(\g-1)\mid \g\in\G, \g\neq 1\}+1.$$ For  $m\in \Z_{\geqslant 0}$ we define $s_m\in \Z_{>0}$ to be the smallest integer such that $v(\sigma^{s_m}-\g_m)\geqslant m$ for some $\g_m\in\G$; in general, $s_m$ might not exist, but $s_0$ certainly does, and we will show in Lemma \ref{lem:sm} that for $m>0$ either none of the $s_m$ exist or all do depending on whether or not $f$ is coseparable. Write $s:=s_N$ and $\tilde{\gamma}:=\g_N$.
\end{notation} 

\begin{hypothesis}{\textup{\textbf{(H3)}}}\label{h3} \textup{[\ref{h2}]} 
Let $m\in\Z_{\geqslant 0}$. If $s_m$ exists, then
\begin{equation*}
\exp\left(\frac{1}{|\Gamma_m|}\sum_{\substack{n\geqslant 1\\ \g\in\G_m}} \deg(\sigma^{s_mn}-\g\g_m^n)\frac{z^n}{n}\right)\in\C(z).
\end{equation*}
\end{hypothesis} 

\begin{remark} 
The statement of Hypothesis \ref{h3} a priori depends on the choice of the elements $\gamma_m$. However, it will follow from Lemma \ref{lem:valsubseq}\ref{bbb} below that it is independent of such a choice. 
\end{remark}

\begin{hypothesis}{\textup{\textbf{(H4)}}}\label{h4} \textup{[\ref{h2}]} 
The number $s$ exists and the sequence
\begin{equation}\label{eq:domrootseq}
\left(\deg(\sigma^{sn}-\tilde{\g}^n)\right)_{n\geqslant 1}
\end{equation}
is a linear recurrent sequence satisfying the dominant root assumption. 
\end{hypothesis} 

\begin{remark} If $s$ exists and \ref{h3} holds, then the sequence \eqref{eq:domrootseq} is automatically linear recurrent. Moreover, by Lemma \ref{lem:sm}, $s$ exists if and only if $f$ is not coseparable, and the element $\tilde{\gamma}\in\G$ is then unique.
\end{remark}

We then have the following results:

\begin{theorem}\label{thm:nb}
Assume $f \colon V\to V$ is a dynamically affine map satisfying the hypotheses \ref{h1}--\ref{h4}. Then $\zeta_f(z)$ is not holonomic. More precisely, it is a product of a root-rational function and a function  admitting a natural boundary along its circle of convergence.
\end{theorem}

\begin{theorem}\label{thm:tame}
Assume $f \colon V\to V$ is a dynamically affine map satisfying the hypotheses \ref{h1}--\ref{h3}. Then $\zeta_f^*(z)$ is root-rational.
\end{theorem}

\noindent The proofs of these theorems will be given in the next section.

\section{Proofs of Theorems \ref{thm:nb} and \ref{thm:tame}} \label{section:proof3.3-4}
\subsection*{Preliminary results on the action of \texorpdfstring{$\Gamma$}{ }} \label{pra}

\begin{lemma}\label{lem:silvertwist}
Let $f \colon V\to V$ be a dynamically affine map.
\begin{enumerate}[\textup{(}i\textup{)}]
\item\label{state:autotwist} There exists a group automorphism $\alpha\colon \G\to\G$ such that for any $\g\in\G$, $\psi\g=\alpha(\g)\psi$ and $\sigma\g=\alpha(\g)\sigma$.
\item\label{state:dynamicallyconfined} 
The map $\sigma^n-\g$ is an isogeny for all $n\in\Z_{>0}$ and $\g\in\G$.
\item\label{state:psikersig} 
$\#(\psi^n-\g)^{-1}(0)=\#(\sigma^n-\g)^{-1}(0)$ for all $n\in\Z_{>0}$ and $\g\in\G$.
\end{enumerate}
\end{lemma}
\begin{proof} 

\ref{state:autotwist} That $\alpha$ exists as a map of sets follows from \cite[Prop.\ 6.77(a)(b)]{Silverman}. Recall that, by assumption, $\sigma$ is surjective and has finite kernel. Now, for all $\g_1,\g_2\in\G$ we have $$\alpha(\g_1\g_2)\sigma=\sigma(\g_1\g_2)=(\sigma\g_1)\g_2=\alpha(\g_1)\alpha(\g_2)\sigma,$$ which implies that $\alpha$ is a group homomorphism. 
For $\g\in\ker(\alpha)$, we have $\sigma(\g-1)=0$, and so $\im(\g-1)\subseteq\ker(\sigma)$. Since $\ker(\sigma)$ is finite and $G$ is connected, we must have $\im(\g-1)=\{0\}$, and so $\g=1$. This shows that $\alpha$ is injective, and hence bijective.

\ref{state:dynamicallyconfined} Let $\g\in\G$ and $n\in\Z_{>0}$. We will show that $\sigma^n-\g$ has finite kernel. Suppose that $x\in G(K)$ is such that $\sigma^n(x)=\g(x)$. 
Put $\beta:=\alpha^n$. Then
\begin{equation}
\sigma^{dn}(x)=\left(\beta^{d-1}(\g)\cdots\beta(\g)\g\right)(x).
\end{equation}
Since $\beta$ is injective and $\G$ is finite, there exists $d\in\Z_{>0}$ for which $$\beta^{d-1}(\g)\cdots\beta(\g)\g=1,$$ so that $(\sigma^n-\g)^{-1}(0)\subseteq (\sigma^{dn}-1)^{-1}(0)$. Since by assumption $\sigma$ is confined, we have that $(\sigma^{dn}-1)^{-1}(0)$ is finite, and the desired result follows.

\ref{state:psikersig} For every $n$, there exists $h_n\in G(K)$ such that $\psi^n(g)=\sigma^n(g)+h_n$ for all $g\in G(K)$. We then have
\begin{equation*}
\#(\psi^n-\g)^{-1}(0)=\#(\sigma^n-\g)^{-1}(-h_n)=\#(\sigma^n-\g)^{-1}(0),
\end{equation*}
where in the last equality we use the fact that $\sigma^n-\g$ is an isogeny.
\end{proof}

\begin{proof}[Proof of Lemma \ref{lem:key}]
The proof of \cite[Lemma 2.4]{BridyBordeaux} shows that
\begin{equation*}
f_n=(f|_C)_n+\frac{1}{|\G|}\sum_{\g\in\G}\#(\psi^n-\g)^{-1}(0).
\end{equation*}
The desired result now follows from Lemma \ref{lem:silvertwist}\ref{state:psikersig}.
\end{proof}

\begin{remark}\label{rem:Bridylemform}
The claim in \cite[Lemma 2.4]{BridyBordeaux} that \eqref{eq:key} holds for dynamically affine maps using Silverman's definition (under the additional assumption that $\psi$ is surjective), is incorrect. For example, when $V=G=E\times E$ for an elliptic curve $E$, $\G=\{1\}$, $\sigma=[1]\times [2]$, and $h=(P,0)$ with $P\in E(K)$ a non-torsion point, then $f_n=\psi_n=0$, but $\ker(\sigma^n-1)\supseteq E(K)\times\{0\}$ is infinite for all $n\in\Z_{>0}$. The mistake in the proof is that under the assumptions in Silverman's definition, Lemma \ref{lem:silvertwist}\ref{state:psikersig} 
 does not need to hold (for this one needs part \ref{state:dynamicallyconfined} 
 of the lemma, which is equivalent to $\sigma$ being confined). Nevertheless, in \cite{BridyBordeaux} the result is only applied for $\dim V=1$, where Silverman's definition implies confinedness of $\sigma$, hence none of the other results are affected.
\end{remark}

\subsection*{Preliminary results on valuations}

\begin{proposition} \label{prop:valprop}  Let $R$ denote a (not necessarily commutative) ring with a valuation $v$. Then the following statements hold for all $x,y\in R$ and $n\in\Z_{> 0}$\textup{:}  \begin{enumerate}[\textup{(}i\textup{)}] \item \label{domain} $R$ has no nontrivial zero divisors.
\item \label{charprime} The characteristic of $R$ is either zero or prime.
\item \label{Gident} We have $v(xy-yx)\geqslant v(x-y)$.
\item \label{noncomm} We have $v(x^n-y^n)\geqslant v(x-y)$.
\item \label{comm} Assume that $x$ and $y$ commute, $v(x)=v(y)=0$, and $v(x-y)>0$. Then\textup{:}
\begin{enumerate}[\textup{(}a\textup{)}]
\item \label{zerozero} if $\mathrm{char}(R)=0$ and $v(\Z- \{0\})=0$, then $v(x^n-y^n)=v(x-y)$\textup{;}
\item \label{zerop} if $\mathrm{char}(R)=0$ and $v(p)>0$ for some prime $p$, then if $v(x-y)>v(p)/(p-1)$, we have $v(x^n-y^n)=v(x-y)+v(n)$\textup{;}
\item \label{pp} if $\mathrm{char}(R)=p>0$, then $v(x^n-y^n)=v(x-y)\cdot |n|_p^{-1}$.
\end{enumerate}
\item \label{unbounded} In cases  \ref{zerop}  and \ref{pp} above, if $z \in R$ satisfies $v(z-1)>0$, then $v(z^n-1)$ is unbounded as $n$ ranges over $\Z_{>0}$.
\end{enumerate}
\end{proposition}

\begin{proof}
\ref{domain} Follows directly from the fact that the valuation $v(x)$ of $x$ is infinite if and only if $x=0$.

\ref{charprime} Follows from \ref{domain}.

\ref{Gident} Follows from the formula $xy-yx=(x-y)x-x(x-y)$.

\ref{noncomm} We have $x^n-y^n=(y+(x-y))^n-y^n=y^n-y^n+z$, where $z$ lies in the two-sided ideal of $R$ generated by $(x-y)$, and hence $v(x^n-y^n)=v(z)\geqslant v(x-y)$.

\ref{comm} Since $x$ and $y$ commute, we have
\begin{equation}\label{liftexpform} 
x^n-y^n=n(x-y)y^{n-1}+\sum_{k=2}^n\binom{n}{k}(x-y)^ky^{n-k}.
\end{equation}
If $v(n)=0$, then the first term has strictly smaller valuation than the second one, and hence $v(x^n-y^n)=v(x-y)$, proving case \ref{zerozero}, as well as cases \ref{zerop} and \ref{pp} for $p{\nmid}n$. It now suffices to consider \ref{zerop} and \ref{pp} for $n=p$; the general result will then follow by induction on $v_p(n)$. For \ref{zerop}, the assumption on $v(x-y)$ implies that
\begin{equation}
v\left(\!\binom{p}{k}\!\right)+(k-1)v(x-y)>v(p)
\end{equation}
for all $2\leqslant k\leqslant p$. This shows that again in (\ref{liftexpform}) the first term has strictly smaller valuation than the second one, which yields $v(x^p-y^p)=v(x-y)+v(p)$. For \ref{pp}, note that $v(x^p-y^p)=v((x-y)^p)=pv(x-y)$. 

\ref{unbounded} Follows from the formula  (\ref{liftexpform}) with $x=z$, $y=1$, and $n$ an arbitrarily large power of $p$.
\end{proof} 

\begin{remark} \textup{[\ref{h2}]}
If $R$ as above is the endomorphism ring of a connected commutative algebraic group $G$ over $K$ and $v$ is a valuation on $\End(G)$ satisfying \textup{[\ref{h2}]}, then Proposition  \ref{prop:valprop}\ref{charprime} can be made slightly more explicit: the characteristic of $R$ will then be either zero or equal to $p=\Char(K)$. In fact, if $\ell\in\Z_{>0}$ is a prime and $v(\ell)>0$, then  the multiplication-by-$\ell$ map is either zero or an inseparable isogeny, and hence its differential, which on the tangent space at $0$ is given by multiplication by $\ell$, is not an isomorphism. Since the tangent space at $0$ is a $K$-vector space, we must have $p=\Char(K)>0$ and $\ell=p$. This also implies that the prime $p$ found in \ref{comm}\ref{zerop} is equal to $\Char(K)$.
\end{remark}

\begin{remark} The assumption that $x$ and $y$ commute is necessary in Proposition \ref{prop:valprop}\ref{comm}\ref{zerop}.  
Consider the quaternion algebra $\mathbf{H}$ generated over $\Q$  by $i,j$ with $i^2=j^2=-1$ and $ij=-ji$, and let $\mc{O}=\Z+\Z i +\Z j +\Z\frac{1+i+j+k}{2}$ be the ring of Hurwitz quaternions, which is a maximal order in $\mathbf{H}$. Consider the valuation $v$ on $\mc{O}$ corresponding to the prime element $1+i \in \mc{O}$. Put $x=i+4j$ and $y=i$. Then $v(x^2-y^2)=v(-16)=8$, but $v(x-y)+v(2)=6$. The assumption that $x$ and $y$ commute is missing from \cite[Lemma 6.2]{BridyBordeaux}, but  the result is only applied for $y=1$, and so other results in that reference are not affected.
\end{remark}

\noindent Recall that $f\colon V\to V$ is a dynamically affine map with associated data as in diagram \eqref{eq:dam}. Assume that $f$ satisfies \ref{h2}. In order to obtain more information about $f$, we will apply Proposition  \ref{prop:valprop} to the ring $R=\RR$ and the valuation $v$ supplied by \ref{h2}.

\begin{lemma} \label{lem:cosep} \textup{[\ref{h2}]}
If $\sigma$ is coseparable, then $\sigma^n-\g$ is a separable isogeny for all $n\in\Z_{>0}$ and $\g\in\G$.
\end{lemma}
\begin{proof}
Let $\g\in\G$ and $n\in\Z_{>0}$. By Lemma \ref{lem:silvertwist}\ref{state:dynamicallyconfined}, $\sigma^n-\g$ is an isogeny, so it remains to show that it is separable. Applying Proposition \ref{prop:valprop}\ref{noncomm}, we see that 
\begin{equation*}
v(\sigma^n-\g)\leqslant v(\sigma^{|\G|n}-\g^{|\G|})=v(\sigma^{|\G|n}-1)=0,
\end{equation*}
where in the last equality we use that $\sigma$ is coseparable.
\end{proof}

\begin{proposition} \label{prop:cosepcase} \textup{[\ref{h1}--\ref{h3}]}
If $f$ is coseparable, then $\zeta_f(z)$ is rational and $\zeta_f^*(z)$ is root-rational.
\end{proposition}
\begin{proof}
If $f$ is coseparable, then by Lemma \ref{lem:cosep}, $\#\ker(\sigma^n-\g)=\deg(\sigma^n-\g)$ for all $n$ and $\g$. The desired result for $\zeta_f(z)$ then follows by applying Lemma \ref{lem:key} together with \ref{h1} and \ref{h3} with $m=0$. By the general rationality conditions in Section 2, this implies that $\zeta_{f^p}(z)$ is rational as well, and hence the result for $\zeta_f^*(z)$ follows from \eqref{eq:tameexpand}.
\end{proof}

\begin{remark}
Proposition \ref{prop:cosepcase} is false if we drop the assumption of the hypotheses.
In fact, if \ref{h2} and \ref{h3} do not hold, then $\zeta_f(z)$ may even have a natural boundary along its circle of convergence (see Example \ref{example:cosepnotrootrat} below).
\end{remark}

\begin{lemma}\label{lem:valsigma0} \textup{[\ref{h2}]}
If $f$ is not coseparable, then $v(\sigma)=0$.
\end{lemma}
\begin{proof}
If $v(\sigma)>0$, then $v(\sigma^n-1)=0$ for all $n$, contradicting the assumption that $\sigma$ is not coseparable.
\end{proof}

\begin{lemma}\label{lem:sigmagammacommute} \textup{[\ref{h2}]}
Suppose that $n\in\Z_{>0}$ and $\g\in\G$ are such that $v(\sigma^n-\g)\geqslant N$. Then $\sigma^n$ and $\g$ commute.
\end{lemma}
\begin{proof}
Since $N>0$ and $v(\gamma)=0$, we have $v(\sigma)=0$. Let $\alpha\in\Aut(\G)$ as in Lemma \ref{lem:silvertwist}\ref{state:autotwist}, so that $\sigma\g=\alpha(\g)\sigma$. It follows that
\begin{equation*}
N\leqslant v(\sigma^n-\g)\leqslant v(\sigma^n\g-\g\sigma^n)=v((\alpha^n(\g)-\g)\sigma^n)=v(\alpha^n(\g)-\g).
\end{equation*}
We conclude that $\alpha^n(\g)=\g$, and hence $\sigma^n \gamma = \alpha^n(\gamma)\sigma^n=\gamma \sigma^n$.
\end{proof}

We will now prove the announced result on the existence of the numbers $s_m$ defined in Notation \ref{not:sm}.

\begin{lemma}\label{lem:sm} \textup{[\ref{h2}]}
\begin{enumerate}[\textup{(}i\textup{)}]
\item\label{lem:sma} If $f$ is coseparable, then none of the numbers $s_m$ exist for $m>0$.
\item\label{lem:smb} If $f$ is not coseparable, then all of the numbers $s_m$ exist.
\end{enumerate}
\end{lemma} 
\begin{proof}
\ref{lem:sma} If $f$ is coseparable, then by Lemma \ref{lem:cosep} all the maps $\sigma^n-\gamma$ for $n\in \Z_{>0}$ and $\gamma\in \Gamma$ are separable isogenies, and hence $v(\sigma^n-\gamma)=0$. Hence $s_m$ do not exist for $m>0$.

\ref{lem:smb} Since $f$ is not coseparable, there is some $n\in\Z_{>0}$ such that $v(\sigma^n-1)>0$, and hence by Proposition \ref{prop:valprop}.\ref{unbounded} the values of $v(\sigma^n-1)$ can be arbitrarily large. This proves the existence of $s_m$ for all $m$.
\end{proof}

\begin{lemma}\label{lem:valsubseq} \textup{[\ref{h2}]}
Suppose that $f$ is not coseparable. Let $m\in\Z_{\geqslant 0}$. Then:
\begin{enumerate}[\textup{(}i\textup{)}]
\item \label{aaa} The set $$S_m:=\{n\in\Z_{>0}\mid  v(\sigma^n-\g)\geqslant m \mbox{ for some } \g\in\G\}$$ is equal to $s_m\Z_{>0}$.
\item \label{bbb} For every $n\in\Z_{>0}$, $\{\g\in\G\mid v(\sigma^{s_mn}-\g)\geqslant m\}=\G_m\g_m^n$.
\end{enumerate}
\end{lemma}
\begin{proof}
\ref{aaa} By Proposition \ref{prop:valprop}\ref{noncomm}, we have $v(\sigma^{s_mn}-\g_m^n)\geqslant v(\sigma^{s_m}-\g_m)\geqslant m$, so $S_m\supseteq s_m\Z_{>0}$. Now suppose to the contrary that there exists an $n\in S_m- s_m\Z_{>0}$. Then there exists a $\g\in\G$ such that $v(\sigma^n-\g)\geqslant m$, and we can write $n=ds_m+r$ for $0<r<s_m$. We obtain 
\begin{equation*}
m\leqslant v(\sigma^n-\g)=v(\sigma^r(\sigma^{ds_m}-\g_m^d)+(\sigma^r-\g\g_m^{-d})\g_m^d).
\end{equation*}
This leads to a contradiction since $v(\sigma^r(\sigma^{ds_m}-\g_m^d))\geqslant m$ and $v((\sigma^r-\g\g_m^{-d})\g_m^d)<m$.

\ref{bbb} We know that $v(\sigma^{s_mn}-\g_m^n)\geqslant m$, so for any $\g\in\G$ we have the equivalence
\begin{align*}
v(\sigma^{s_mn}-\g)\geqslant m & \iff v(\g-\g_m^n)\geqslant m \\ & \iff v(\g\g_m^{-n}-1)\geqslant m\iff \g\in\G_m\g_m^n. \qedhere
\end{align*}
\end{proof}

\subsection*{Preliminary results on natural boundaries}

\begin{lemma}\label{lem:NBBasic}
Let $h,\beta \in\R_{>0}$ with $\beta<1$. Then the power series
\begin{equation}
G_h(z):=\sum_{n\geqslant 1}|n|_p^hz^n,\qquad H_\beta(z):=\sum_{n\geqslant 1}\beta^{|n|_p^{-1}}z^n
\end{equation}
have radius of convergence $1$ and define holomorphic functions that have a natural boundary along the unit circle.
\end{lemma}
\begin{proof}
That the radius of convergence is $1$ follows from the fact that
\begin{equation*}
\displaystyle\limsup_{n\to \infty} \big(|n|_p^h\big)^{1/n} = 1= \displaystyle\limsup_{n\to \infty} \big(\beta^{|n|_p^{-1}}\big)^{1/n}.
\end{equation*}
Now, note that $G_h$ and $H_{\beta^p}$ satisfy the following similar functional equations:
\begin{eqnarray}
G_h(z) &=& \frac{z}{1-z}-\frac{z^p}{1-z^p}+p^{-h}G_h(z^p),\label{eq:Ghfunc}\\
H_{\beta}(z) &=& \beta\left(\frac{z}{1-z}-\frac{z^p}{1-z^p}\right)+H_{\beta^p}(z^p).\nonumber
\end{eqnarray}
In order to prove the statement on the natural boundary, we will show by induction on $k\geqslant 1$ that for every primitive $p^k$-th root of unity $\omega$ we have
\begin{equation*}
\lim_{\lambda\to 1^{-}}G_h(\lambda\omega)=-\infty=\lim_{\lambda\to 1^{-}}H_{\beta}(\lambda\omega).
\end{equation*}
We present details for the case of $G_h(z)$; the proof for $H_{\beta}(z)$ is analogous. For $k=1$, it follows from \eqref{eq:Ghfunc} that for every $0<\lambda<1$ we have
\begin{equation*}
G_h(\lambda\omega)=\frac{\lambda\omega}{1-\lambda\omega}-\frac{\lambda^p}{1-\lambda^p}+p^{-h}G_h(\lambda^p).
\end{equation*}
As $G_h(\lambda^p)\leqslant\lambda^p/(1-\lambda^p)$ and $h>0$, it follows that $G_h(\lambda\omega)\to-\infty$ as $\lambda\to 1^{-}$. For $k>1$, the result follows from induction by substituting $\lambda\omega$ into \eqref{eq:Ghfunc}.
\end{proof}

\begin{remark}
Alternatively, since \eqref{eq:Ghfunc} implies that $G_h(z)$ is a so-called $p$-Mahler function, we could have immediately concluded that $G_h(z)$ is either rational or has the unit circle as a natural boundary by a result of Rand\'e \cite{Rande} (see also \cite[Thm.\ 2]{BCR}). The former possibility can be excluded by an explicit computation using the functional equation \eqref{eq:Ghfunc}. Such an  approach does not work for $H_{\beta}(z)$.
\end{remark}

\subsection*{Proofs of Theorems \ref{thm:nb} and \ref{thm:tame}}
 Assume that $f$ satisfies \ref{h1}--\ref{h3}. We have already dealt with the case  where $f$ is coseparable in Proposition \ref{prop:cosepcase}, so it remains to consider the case where $f$ is not coseparable.

Using Lemma \ref{lem:key}, we may write the zeta function $\zeta_f(z)$  as

$$\zeta_f(z)=\zeta_{f|_C}(z) \cdot \exp \left(\sum_{n\geqslant 1}\sum_{\gamma\in \Gamma} \frac{\deg(\sigma^n-\gamma)}{p^{v(\sigma^n-\gamma)}}\frac{z^n}{n}\right)^{1/|\Gamma|},$$ 
with a similar expression for the tame zeta function $\zeta_f^*(z)$. For  $m\in \Z_{\geqslant 0}$, we consider separately the terms corresponding to a fixed value of $v(\sigma^n-\gamma)$, giving rise to functions

$$\zeta_{f,m}(z)= \exp \left( \sum_{\substack{ n \geqslant 1, \g\in\G\\ v(\sigma^n-\g)=m}}  \frac{\deg(\sigma^n-\gamma)}{p^m
}\frac{z^n}{n}\right).$$ Consider the sets $$T_m=\{(n,\gamma)\in \Z_{>0}\times\:\Gamma \mid v(\sigma^n-\gamma)\geqslant m\}.$$ By Lemma \ref{lem:valsubseq}, we have $T_m=\{(s_m n,\gamma \gamma_m^n)\mid n\in \Z_{>0}, \gamma\in \Gamma_m\}$, and hypothesis \ref{h3} implies that the function $$F_m(z)=\exp \left(\sum_{(n,\gamma)\in T_m} \deg(\sigma^n-\gamma) \frac{z^n}{n}\right)$$ satisfies $F_m(z)^{s_m/|\G|}\in\C(z^{s_m})$; hence it is root-rational. It follows that the function $$\zeta_{f,m}(z)=(F_m(z)/F_{m+1}(z))^{1/p^m}$$ is root-rational as well.

We analogously define the tame functions $\zeta^*_{f,m}(z)$, summing only over indices $n$ coprime to $p$. By Equation \eqref{eqn:finaltame}, these are also root-rational, and we have the product formulas \begin{equation}\label{eq:rootratexpand}\zeta_f(z)=\zeta_{f|_C}(z)\left(\prod_{m\geqslant 0}\zeta_{f,m}(z)\right)^{1/|\G|}  \mbox{and }  \zeta^*_f(z)=\zeta^*_{f|_C}(z)\left(\prod_{m\geqslant 0}\zeta^*_{f,m}(z)\right)^{1/|\G|}\!\!\!\!\!\!\!\!\!. \end{equation}

Our next aim is to simplify the tail (i.e.\ the product of all terms with $m$ suitably large) of \eqref{eq:rootratexpand}  using Proposition \ref{prop:valprop}\ref{comm}. To this end, take an integer $M\geqslant \max(N,v(p)/(p-1)+1)$ and set $r:=s_M$ and $\tau:=\sigma^{s_M}\g_M^{-1}$. By Lemma \ref{lem:sigmagammacommute} the elements $\sigma^{s_M}$ and $\gamma_M$ commute, and we can rewrite the tail of \eqref{eq:rootratexpand} as
\begin{equation*}
\prod_{m\geqslant M}\zeta_{f,m}(z)=\exp\left(\sum_{n\geqslant 1}\frac{\deg(\tau^n-1)}{p^{v(\tau^n-1)}}\frac{z^{rn}}{rn}\right).
\end{equation*}
Set $C:=v(\tau-1)\geqslant M$. By Proposition \ref{prop:valprop}\ref{comm} applied to $x=\tau$ and $y=1$, we obtain
\begin{equation}
v(\tau^n-1)=\begin{cases}
C+v(n) &\mbox{if } \Char(\End(G))=0;\\
C|n|_p^{-1} &\mbox{if } \Char(\End(G))=p>0.
\end{cases}
\end{equation}
In particular, if $p{\nmid} n$, then $v(\tau^n-1)$ is independent of $n$, so we have $\zeta_{f,m}^*(z)=1$ for $m\geqslant M$ and $m\neq C$. The product expansion in \eqref{eq:rootratexpand} therefore shows that the tame zeta function is root-rational, proving Theorem \ref{thm:tame}.

Now suppose that $f$ also satisfies \ref{h4}. Since $s$ divides $r$ and $\deg(\tau^n-1)=\deg(\sigma^{rn}-\tilde{\gamma}^{rn/s})$, we see that $(\deg(\tau^n-1))_{n\geqslant 1}$ is a linear recurrent sequence with a unique dominant root, say $\Lambda$, with multiplicity $\mu\in\Z$. We then obtain
\begin{equation*}\label{eqn:itspastdeadline}
z \frac{d}{dz} \log \prod_{m\geqslant M} \zeta_{f,m}(z)=\sum_{n\geqslant 1}\mu\Lambda^nz^{rn}\cdot\begin{cases}
p^{-C} |n|_p^{v(p)} &\mbox{ if } \Char(\End(G))=0\\
p^{-C|n|_p^{-1}} &\mbox{ if } \Char(\End(G))=p
\end{cases}\, \, + R(z),
\end{equation*}
where $R(z)$ is some power series with radius of convergence $>|\Lambda|^{-1/r}$.

As stated in \cite[Lemma 1]{BMW}, the existence of a natural boundary for a series $\sum a_n z^n$ along its circle of convergence implies the existence of a natural boundary along its circle of convergence for the corresponding zeta function $\exp \sum a_n z^n/n$. 
By applying Lemma \ref{lem:NBBasic} with $h=v(p)$ or $\beta=p^{-C}$ (depending on whether $\End(G)$ is of characteristic $0$ or $p$) and substituting $\Lambda z^r$ for $z$ into $G_h(z)$ or $H_{\beta}(z)$, 
it follows that the series $$\left(\prod_{m\geqslant M} \zeta_{f,m}(z)\right)^{1/|\Gamma|}$$ has a natural boundary along its circle of convergence. Theorem \ref{thm:nb} now follows from the product expansion in \eqref{eq:rootratexpand}. \qed

\begin{remark}
An examination of the (by the proof, finite) product expansion for the tame zeta function \eqref{eq:rootratexpand}, shows that in fact $\zeta_f^*(z)^t\in\C(z)$ for $t=p^{C+1}r$, where $C$ and $r$ are as in the proof.
\end{remark}

\section{Discussion of the hypotheses}  \label{sec:disc}

\subsection*{Classification of $G$ for $V=\PP^1$}
Suppose that $V=\PP^1$ and recall that $K$ is algebraically closed. Since $\G$ is finite and $\iota$ has Zariski-dense image, for dimension reasons $G$ is a connected one-dimensional algebraic group. By the Barsotti--Chevalley structure theorem for algebraic groups \cite{Chevalley, ConradChev}, $G$ is an extension of a linear algebraic group by an abelian variety, and thus (again by connectedness and dimension considerations) $G$ is either a one-dimensional connected linear algebraic group or an abelian variety of dimension one. In the latter case, $G$ is an elliptic curve $E$. In the former case, either $G=\Gm$, the multiplicative group; or $G=\Ga$, the additive group \cite[Thm.\ 3.4.9]{SpringerLAG}.
We  denote by $[m]$ the multiplication-by-$m$ map on $G$. The corresponding endomorphism rings are as follows:
\begin{enumerate}[\textup{(}i\textup{)}]
\item if $G=\Gm$, then $\End(G)\cong\Z$, with the map $[m]$ given by $x\mapsto x^m$;
\item if $G=\Ga$, then $\End(G)\cong K\langle \phi\rangle$ is the ring of skew-commutative polynomials in the Frobenius $\phi\colon x\mapsto x^p$, with $\phi a=a^p\phi$ for all $a\in K$;
\item if $G=E$ is  an elliptic curve, then $\End(E)$ is either $\Z$, an order in an imaginary quadratic number field in which $p$ splits, or a maximal order in the quaternion algebra over $\Q$ that ramifies precisely at $p$ and $\infty$ \cite{Deuring}.
\end{enumerate}

\subsection*{Hypothesis \ref{h1}}

\begin{lemma}\label{lem:finiteset}
If $f\colon C\to C$ is an arbitrary map on a finite set $C$, then $\zeta_f(z)$ is rational and $\zeta_f^*(z)$ is root-rational.
\end{lemma}
\begin{proof}
Since there are only finitely many orbits, this follows for $\zeta_f(z)$ from the Euler product \eqref{euler}, and then for $\zeta^*_f(z)$ from \eqref{eq:tameexpand}. 
\end{proof}

\begin{corollary}\label{cor:h1p1}
If $f\colon V \to V$ is a dynamically affine map and either $\dim V=1$ \textup{(}e.g.\ if $V=\PP^1$\textup{)} or $G$ is complete \textup{(}e.g.\ if $G$ is an abelian variety\textup{)}, then $f$ satisfies \ref{h1}.
\end{corollary}
\begin{proof}
If $G$ is complete, then $C=\emptyset$, and the result is clear. If $\dim V=1$, then the assumption that $\Gamma\backslash G $ is a Zariski-dense open subset of $V$ implies that $C$ is finite, and Lemma \ref{lem:finiteset} applies.
\end{proof}

\begin{example}
We give an example where \ref{h1} fails. Consider $G=\Gm\tm\Gm$, $\G=\{1\}$, the standard embedding $G\hookrightarrow V:=\PP^1\tm\PP^1$ and $$f\colon V\to V,\: (x,y)\mapsto (x^m,y^m)$$ for an integer $m\geqslant 2$ coprime to $p$. Then 
\begin{equation*}
C=(\PP^1\tm\{0\})\cup(\PP^1\tm\{\infty\})\cup(\{0\}\tm\PP^1)\cup(\{\infty\}\tm\PP^1)
\end{equation*}
is a union of four copies of $\PP^1$ intersecting in four points $$C'=\{(0,0),(0,\infty),(\infty,0),(\infty,\infty)\},$$ all of which are fixed by $f$. Applying Remark \ref{rem:zetaMV}, we get that \begin{equation*}
\zeta_{f|_C}(z)=\zeta_g(z)^4(1-z)^4,
\end{equation*}
where $\zeta_g(z)$ is the zeta function of $g\colon{\PP^1}\to\PP^1,\: x\mapsto x^m$, which was shown to be transcendental over $\C(z)$ by Bridy \cite{Bridy} (our result in this paper even shows that the function has a natural boundary).
\end{example}

\subsection*{Hypothesis \ref{h2}}

\begin{proposition}\label{prop:h2general} \mbox{ }
\begin{enumerate}[\textup{(}i\textup{)}]
\item \label{state:as} A nontrivial abelian variety $A$ satisfying \ref{h2} with $\RR=\End(A)$ has to be simple.
\item \label{state:gns} There exist commutative algebraic groups $G$ of arbitrary dimension $>1$ that satisfy \ref{h2} with $\RR=\End(G)$ but that are not simple.
\item \label{state:h2id} For any connected commutative algebraic group $G$, \ref{h2} is equivalent to the claim that all nonzero elements of $\RR$ are isogenies and for every $m$, the set 
$$I_m := \{ \tau \in \RR - \{0\} \mid \log_p \deg_{\mathrm{i}}(\tau) \geqslant m \} \cup \{ 0 \}$$
is an ideal in $\RR$. \textup{(}Note that by the multiplicativity of the inseparable degree this is equivalent to 
$ \deg_{\mathrm{i}}(\tau_1 + \tau_2) \geqslant \min \{  \deg_{\mathrm{i}}(\tau_1),  \deg_{\mathrm{i}}(\tau_2) \} $ for all nonzero $\tau_1, \tau_2$, $\tau_1\neq -\tau_2$.\textup{)}
\item \label{state:orda} Let $G=A$ be a nontrivial abelian variety. Then \ref{h2} holds with $\RR=\Z \hookrightarrow \End(A)$ \textup{(}and then necessarily $\Gamma =\{1\}$ or $\Gamma=\{\pm 1\}$\textup{)}. 
\end{enumerate}
\end{proposition}
\begin{proof} Note that the hypothesis on the existence of $v$ implies that $\RR$ is a (not necessarily commutative) domain (Proposition \ref{prop:valprop}\ref{domain}). 

\ref{state:as} Since an abelian variety $A$ factors up to isogeny into a direct product of simple abelian varieties, $\End(A)$ is a domain if and only if $A$ is simple. 

\ref{state:gns} Consider extensions of algebraic groups
$$ 1 \rightarrow \Gm \rightarrow G \rightarrow A \rightarrow 1,$$ where $G$ is abelian and $A$ is any simple abelian variety. 
These are classified by $\mathrm{Ext}^1(A,\Gm) \cong \widehat A(K)$, where $\widehat A$ is the dual abelian variety of $A$ \cite[Thm.\ 9.3]{MilneAV}. Suppose $\widehat A(K)$ has a non-torsion point $P$ (in particular, $K$ has to be transcendental over $\F_p$) and choose an extension corresponding to $P$. We claim that $G$ does not contain any nontrivial abelian variety. Suppose otherwise and let $A'$ be a nontrivial abelian variety contained in $G$. The image of $A'$ in $A$ cannot be zero, and hence is equal to all of $A$ since $A$ is simple. It follows that $A'$ and $\Gm$ generate $G$, and a result of Arima \cite[Thm.~2]{Arima} implies that the extension corresponds to a point of finite order in $\widehat A(K)$. We conclude that $G$ does not contain any nontrivial abelian variety, and hence by \cite[Prop.~7]{Arima} the restriction map $\End(G) \rightarrow \End(\Gm) \cong \Z$ is injective, meaning that $\End(G)\cong \Z$. The inseparable degree of $[n]$ on $G$ is the product of those on $\Gm$ and $A$, and if $A$ has dimension $g$ and $p$-rank $r$, then the valuation $v=(2g-r+1)v_p$ on $\RR=\End(G)$ satisfies \ref{h2}. 

\ref{state:h2id} If $v$ is a valuation as in \ref{h2}, then $$I_m = \{ \tau \in \RR \mid v(\tau) \geqslant m \}$$ is an ideal. Conversely, if all $I_m$ are ideals, then $v$ defined by $$v(\tau):= \sup \{ m \mid \tau \in I_m \}$$ satisfies \ref{h2}. 

\ref{state:orda} For a nontrivial abelian variety $A$ the endomorphism ring $\End(A)$ has characteristic zero and the maps $[m] \colon A \rightarrow A$ for $m \in \Z\backslash\{0\}$ are isogenies, and are separable if and only if $p{\nmid} m$. By multiplicativity of the inseparable degree, we find that $\deg_{\mathrm{i}} ([m]) = \deg_{\mathrm{i}}([p])^{v_p(m) }$, and hence the valuation $v \colon \RR \rightarrow \Z\cup \{\infty\}$ given by $v([m])= c v_p(m)$ with $c:= \log_p \deg_{\mathrm{i}}([p]) $ satisfies \ref{h2}. \qedhere
\end{proof}

\begin{lemma}\label{lem:h2p1}
Hypothesis \ref{h2} holds for dynamically affine maps on $\PP^1$.
\end{lemma}
\begin{proof}
We verify, for $G$ a one-dimensional connected algebraic group, that the $I_m$ as in Proposition \ref{prop:h2general}\ref{state:h2id} are indeed ideals. Note that the claim that all nonzero elements of $\End(G)$ are isogenies is immediate by a dimension argument. For $G=\Gm$ and $G=\Ga$, the set of inseparable isogenies together with the zero map is the principal ideal generated by the Frobenius $\phi\colon x\mapsto x^p$, so $I_m=(\phi^m)$ is an ideal. If $G=E$ is an elliptic curve, then for any isogeny $\tau\colon E\to E$, we have that $\log_p\deg_{\mathrm{i}}(\tau)$ is the largest $r>0$ for which $\tau$ factors through the $p^r$-Frobenius $E\to E^{\left(p^r\right)}$ \cite[II.2.12]{Silverman2}, which again implies that the $I_m$ are ideals.
\end{proof}

\begin{remark} Another approach, folllowing \cite{BridyBordeaux},  is to check the result for each of the possible one-dimensional groups $G$ with the following valuations: 
\begin{enumerate}[(i)]
\item \label{gmautv} if $G=\Gm$, then on $\End(G) \cong \Z$ set $v=v_p$, the $p$-adic valuation;
\item \label{gaautv} If $G=\Ga$, then on $\End(G)=K\langle\phi\rangle$ set $v=v_\phi$, the valuation associated to the two-sided ideal $(\phi)$;
\item \label{eautv} If $G=E$ is an elliptic curve, set $v=v_p \circ N$, where $N$ is the  field norm of the extension $\End(E)\otimes \Q$ of ${\Q}$ if $E$ is ordinary and $N$ is the reduced norm on the quaternion algebra $\End(E) \otimes \Q$ if $E$ is supersingular.
\end{enumerate}
\end{remark}

\subsection*{Hypothesis \ref{h3}}

The following general observation will be used multiple times to verify that Hypothesis \ref{h3} holds in certain cases: \emph{If $R$ is a (not necessarily commutative) domain and $\Gamma$ is a nontrivial finite subgroup of the multiplicative group of $R$, then $\sum_{\gamma \in \Gamma} \gamma =0$.}

We first discuss the degree function on a commutative subring $\RR$ of the endomorphism ring $\End(A)$ of an abelian variety $A$. 
The ring $S=\End(A)\otimes_{\Z} \overline{\Q}$ is a semisimple $\overline{\Q}$-algebra, and hence  is isomorphic to a product of finitely many (full) rings of matrices over $\overline{\Q}$. Let  $$\psi=(\psi_1,\dots,\psi_k) \colon S \to \prod_{i=1}^k \mathrm{M}_{n_i}(\overline{\Q})$$ be such an isomorphism. The degree of an endomorphism $\alpha\in \End(A)$ can be computed by the formula   $$\deg(\alpha)= \prod_{i=1}^k \det\psi_i(\alpha)^{\nu_i},$$ where $\nu_i\in \Z_{>0}$ are certain integers (see \cite[Cor.~3.6]{Grieve} or the discussion in \cite[Prop.~2.3]{D1}). Since the ring $\RR$ is commutative, the matrices in $\psi_i(\RR)\subseteq \mathrm{M}_{n_i}(\overline{\Q})$ can be simultaneously triangularised, so that after conjugating by appropriate matrices,  we may assume that $\psi(\RR)$ lies in the product of rings $\mathrm{UT}_{n_i}(\overline{\Q})$ of upper triangular matrices. Composing the homomorphism $\psi_i|_{\RR}$ with the homomorphism $\mathrm{UT}_{n_i}(\overline{\Q}) \to  \overline{\Q}^{\, n_i}$ that maps each matrix to the tuple consisting of its diagonal elements, we obtain a ring homomorphism $$\lambda=(\lambda_1,\ldots,\lambda_l) \colon \RR \to \overline{\Q}^{\, l},$$ where $l=\sum_{i=1}^k n_i$. The degree function on $\RR$ then takes the form $$\deg(\alpha)= \prod_{j=1}^l \lambda_j(\alpha)^{\mu_j},$$ where $\mu_j\in \Z_{>0}$ are certain integers.

\begin{proposition}\label{prop:abvarh3}
Let $A$ be an abelian variety over $K$, let $f$ be a dynamically affine map with $G=A$, and  let $\RR$ be a commutative subring of the endomorphism ring $\End(A)$ that contains both $\sigma$ and $\Gamma$. If  \ref{h2} is satisfied for the ring $\RR$, then $f$ satisfies \ref{h3}.
\end{proposition}
\begin{proof}
Using the notation provided by the statement of \ref{h3}, write $\tau:=\sigma^{s_m}$ and $\tilde{\tau}:=\tau\g_m^{-1}$. By Lemma \ref{lem:silvertwist}\ref{state:dynamicallyconfined}, confinedness of $\sigma$ implies  that $\tau^n-\g$ is an isogeny for all $\g\in\G$. Since $\RR$ is commutative, we have $\deg(\tau^n-\g\g_m^n)=\deg(\tilde\tau^n-\g)$ for $\g\in \Gamma$.

Using the notation explained at the beginning of this subsection, 
we obtain the formula
$$\sum_{\g\in\G_m}\deg(\tilde\tau^n-\g) = \sum_{\g\in\G_m} \prod_{j=1}^l \lambda_j(\tilde\tau^n-\g)^{\mu_j}.$$
Since $\lambda_j \colon \RR \to \overline{\Q}$ are ring homomorphisms, we may expand the product on the right hand side and rewrite the formula as 
$$\sum_{\g\in\G_m} \prod_{j=1}^l \lambda_j(\tilde\tau^n-\g)^{\mu_j}=\sum_{k}\sum_{\g\in\G_m}\chi_k(\g)\eta_k^n$$
for some $\eta_k\in\overline{\Q}$ and some characters $\chi_k\colon \Gamma\to\overline{\Q}^{\tm}$. If a character $\chi_k$ is nontrivial, then $\sum_{\g\in\G_m}\chi_k(\g)=0$; otherwise, $\sum_{\g\in\G_m}\chi_k(\g)=|\G_m|$. Thus, we obtain
\begin{equation*}
\frac{1}{|\G_m|}\sum_{\g\in\G_m}\deg(\tau^n-\g\g_m^n)=\sum_{\substack{\chi_k=1}}\eta_k^n,
\end{equation*}
and the desired result follows from the general criteria for rationality of the zeta function discussed in Section 2.
\end{proof}

\begin{proposition}\label{prop:h3p1}
A dynamically affine map on $\PP^1$ satisfies \ref{h3}.
\end{proposition}
\begin{proof}
As in the proof of the previous proposition, we set $\tau:=\sigma^{s_m}$. If $G=\Gm$, then, identifying $\End(G)$ with $\Z$, we have
\begin{equation*}
\deg(\tau^n-\g\g_m^n)=|\tau^n-\g\g_m^n|=\deg(\tau)^n-\g\g_m^n\sgn(\tau)^n.
\end{equation*}
If $G=\Ga$, then 
\begin{equation*}
\deg(\tau^n-\g\g_m^n)=\deg(\tau)^n.
\end{equation*}
(This holds even when $\deg(\tau)=1$ since $\tau$ is confined.) Finally, if $G=E$ is an elliptic curve, then $\deg(\tau)=\tau\overline{\tau}$, where $\overline{\tau}$ denotes $\tau$ or the complex/quaternionic conjugate of $\tau$ depending on whether $\End(E)$ is $\Z$, an order in an imaginary quadratic number field or an order in a quaternion algebra over $\Q$. In either case
\begin{align*}
\sum_{\g\in\G_m}&\deg(\tau^n-\g\g_m^n) = \sum_{\g\in\G_m}\big((\tau\overline{\tau})^{\, n}-\g\g_m^n\overline{\tau}^{\, n}-\tau^n\overline{\g_m}^{\, n}\overline{\g}+1\big)\\
&= |\G_m|(\deg(\tau)^n+1)-\left(\sum_{\g\in\G_m}\g\right)\g_m^n\overline{\tau}^{\, n}-\tau^n\overline{\g_m}^{\, n}\left(\sum_{\g\in\G_m}\overline{\g}\right).
\end{align*}
Combining the three cases, we find that in case $\G_m$ is nontrivial, we get
\begin{equation}\label{eq:p1casesnontriv}
\sum_{\g\in\G_m}\deg(\tau^n-\g\g_m^n)=\begin{cases}
|\G_m|\deg(\tau)^n &\mbox{ if } G=\Gm \mbox{ or } G=\Ga;\\
|\G_m|(\deg(\tau)^n+1) &\mbox{ if } G=E,
\end{cases}
\end{equation}
whereas in case $\G_m$ is trivial, the elements $\tau$ and $\g_m$ commute by Lemma \ref{lem:sigmagammacommute}, and hence
\begin{align}\label{eq:p1casestriv}
\sum_{\g\in\G_m} \deg(\tau^n-\g\g_m^n) & =   \deg(\tau^n-\g_m^n) \\ & =\begin{cases}
\deg(\tau)^n-(\g_m\sgn(\tau))^n &\mbox{ if } G=\Gm;\\
\deg(\tau)^n &\mbox{ if } G=\Ga;\\
\deg(\tau)^n+1-(\g_m\overline{\tau})^n-(\tau\overline{\gamma_m})^n &\mbox{ if } G=E.
\end{cases} \nonumber
\end{align}
We may regard these formulas as equalities between complex numbers (for $G=E$ embedding the field $\Q(\gamma_m \overline{\tau})$ into $\C$).  It follows that the corresponding zeta function is rational.
\end{proof}

\begin{example}\label{example:cosepnotrootrat}
For an example where \ref{h3} does not hold, consider $K=\overline \F_3$, $G=\Ga\tm\Ga$, $\Gamma=\{1\}$, $V=G$  and 
$$ f \colon G\to G, (x,y) \mapsto (x^9+y^3, x^3). $$ 
Since the differential of $\sigma=f$ is zero, the map $f$ is even coseparable. One may directly compute the values of $\deg(\sigma^n-1)$. One way to do this is to write $$\sigma = \begin{pmatrix} \phi^2& \phi \\ \phi& 0\end{pmatrix}$$ with $\phi\colon \Ga \to \Ga$ the Frobenius map, and show that for a matrix $\tau$ in $\mathrm{M}_2(\F_3[\phi])$ with nonzero determinant the degree of $\tau$ as a map $\tau\colon G \to G$ and the degree in $\phi$ of $\det(\tau)\in \F_3[\phi]$ are related by the formula $$\deg(\tau)=3^{\deg_{\phi}(\det(\tau))}.$$ (This follows easily by writing $\tau$ in Smith normal form.) Computing the eigenvalues of $\sigma$ as Laurent series in $\phi^{-1}$, we get that $$\deg(\sigma^n-1) = \begin{cases} 9^n, & 2{\nmid} n;\\ 9^{n- |n|_3^{-1}},& 2{\mid}n,\end{cases}$$ and hence the zeta function satisfies the equation  
$$z\frac{d}{dz}\log \zeta_f(z)=\sum \deg(\sigma^n-1) z^n = \frac{9z}{1-81z^2} + H_{1/9} (81 z^2), $$
where $H_{1/9}(z)$ is the function from Lemma \ref{lem:NBBasic}. It follows that $\zeta_f(z)$ has a natural boundary along $|z|=1/9$ and \ref{h3} indeed fails to hold.

For a detailed computation of the degree and a general discussion of fixed points of endomorphisms of vector groups, we refer the reader to \cite{BCH}.
\end{example}

\subsection*{Hypothesis \ref{h4}}

\begin{proposition}\label{prop:h4p1}
A dynamically affine non-coseparable map on $\PP^1$ satisfies \ref{h4}.
\end{proposition}
\begin{proof}
It follows directly from \eqref{eq:p1casesnontriv} and \eqref{eq:p1casestriv} applied for $m=N$ that $\deg(\tau)$ is the dominant root (note that for $G=\Gm$ and $G=E$, confinedness of $\sigma$ implies that $\deg(\tau)\geqslant 2$).
\end{proof}

We will now examine property \ref{h4} in the case where $G=A$ is an abelian variety.

\begin{proposition}\label{prop:h4av} Let $f$ be a dynamically affine map with $G=A$ an abelian variety. Assume that the hypothesis \ref{h2} is satisfied for a commutative ring $\RR\subseteq \End(A)$.  \begin{enumerate}[\textup{(}i\textup{)}] \item\label{prop:h4av1} The map $f$ satisfies $\ref{h4}$ if and only if $\sigma$ is not coseparable and the characteristic polynomial of the action of $\sigma$ on the $\ell$-adic Tate module $T_{\ell}(A)$, where $\ell$ is any prime $\ell \neq p$, has no roots of complex absolute value $1$.
\item\label{prop:h4av2} The map $f$ satisfies \ref{h4} if and only if the map $\sigma$, regarded as a dynamically affine map $\sigma\colon A\to A$, satisfies \ref{h4}.
\end{enumerate}
\end{proposition}
\begin{proof} By Lemma \ref{lem:sm} it suffices to treat the case where $\sigma$ is non-coseparable and the number $s$ exists.

\ref{prop:h4av1} We have the formula \begin{equation}\label{eqn:prop:h4av} \deg(\sigma^{sn}-\tilde{\gamma}^n)= \prod_{j=1}^l \left(\lambda_j(\sigma)^{sn}-\lambda_j(\tilde{\gamma})^n\right)^{\mu_j}.\end{equation} Since $\Gamma$ is finite and $\sigma$ is confined, the elements $\lambda_j(\tilde{\gamma})$ are roots of unity while $\lambda_j(\sigma)$ are not. By the discussion in \cite[Section 5]{D1}, the elements $\lambda_j(\sigma)$ are exactly the roots of the action of $\sigma$ on $T_{\ell}(A)$. Expanding the expression in (\ref{eqn:prop:h4av}), one can show that $\deg(\sigma^{sn}-\tilde{\gamma}^n)$ is given by a linear recurrence, and that the dominant root is unique if and only if $|\lambda_j(\sigma)|\neq 1$ for all $j$. (The argument is identical to the one where $\tilde{\gamma}=1$ as given in \cite[Prop.~5.1(v)]{D1}.)

\ref{prop:h4av2} This is clear since the condition in \ref{prop:h4av1} depends neither on $s$ nor on $\tilde{\gamma}$.
\end{proof}

\begin{example}
For an example where \ref{h4} fails, let $G=\Gm^4$, $\G=\{1\}$ and $V=G$. We choose $f=\sigma\in\End(G)\cong \mathrm{M}_4(\Z)$ to be the companion matrix of the minimal polynomial $g$ of a Salem number $\alpha>1$ of degree $4$ (e.g.\ $g=x^4-3x^3+3x^2-3x+1$ \cite{Smyth}). Then $\deg(\sigma^n-1)=|\det(\sigma^n-1)|$ and if $\beta\in\C$ is a zero of $g$ with absolute value $1$, then $\alpha$ and $\alpha\beta$ are distinct dominant roots of the linear recurrent sequence $(\deg(\sigma^n-1))_{n\geqslant 1}$.
\end{example}

\subsection*{Proofs of Theorems \ref{thm:main} and \ref{higherdim}} 

\begin{proof}[Proof of Theorem \ref{thm:main}] The theorem will follow by combining Theorems \ref{thm:nb} and \ref{thm:tame} with the following observations. 

A dynamically affine map $f\colon\PP^1 \to \PP^1$ satisfies the hypotheses \ref{h1}--\ref{h3} by Corollary \ref{cor:h1p1}, Lemma \ref{lem:h2p1} and Proposition \ref{prop:h3p1}. If $f$ is not coseparable, it satisfies \ref{h4} by Proposition \ref{prop:h4p1}. If $f$ is coseparable, the function $\zeta_f(z)$ is rational by Proposition \ref{prop:cosepcase}. \end{proof}

\begin{proof}[Proof of Theorem \ref{higherdim}]
In this situation, \ref{h1}--\ref{h3} hold by Corollary \ref{cor:h1p1} and Propositions \ref{prop:h2general}\ref{state:orda} and \ref{prop:abvarh3}. If $p{\mid}m$, the map $\sigma$ is coseparable, and $\zeta_f(z)$ is rational. If $p{\nmid}m$, $\sigma$ is not coseparable, and $f$ satisfies \ref{h4} by Proposition \ref{prop:h4av}. 
\end{proof} 
 
\appendix

\section{\mbox{ } \\ Radius of convergence of \texorpdfstring{$\zeta_f$}{} for dynamically affine maps \texorpdfstring{$f$}{}} \label{conv} 

In general, the existence of a positive  radius of convergence of a dynamical zeta function is a nontrivial property of the growth of the number of periodic points of a given order. In the manifold setting, this is studied in \cite{AM}; Kaloshin showed that a positive convergence radius is \emph{not} topologically Baire generic in the $C^r$ topology for any $2 \leqslant r < + \infty$ \cite[Corollary 1]{Kaloshin}.

In this appendix, we study this problem for a morphism $f \colon V \rightarrow V$ on an algebraic variety $V$. We can say something in case $f$ is dynamically affine, or in case $V$ is smooth projective, but we do not know what happens in the general case. 

\begin{theorem}  Let $f \colon V \rightarrow V$ denote a dynamically affine map over an algebraically closed field $K$ of characteristic $p$, satisfying \textup{\textbf{(H1)}}. Then the zeta functions $\zeta_f(z)$ and $\zeta^*_f(z)$ converge to holomorphic functions on a nontrivial open disk centred at the origin. \end{theorem} 

\begin{proof} 
It follows from the definitions that $\zeta_f^*(z)$  converges whenever $\zeta_f(z)$ does. The latter function converges for $|z|<1/\limsup \sqrt[n]{f_n}$. Hence to prove the statement, it suffices to prove that $f_n \leqslant c^n$ for some constant $c$. By Lemma \ref{lem:key}, it suffices to prove that $(f|_C)_n \leqslant c^n$ and $\# \ker (\sigma^n-\gamma) \leqslant c^n$ for all $\gamma \in \Gamma$ and some constant $c$ (independent of $n$). The first statement follows immediately from \textup{\textbf{(H1)}}.

For the second statement, we note that $\# \ker (\sigma^n-\gamma)  = \# \mathrm{Fix}(\tau)$, where $\tau = \sigma^n \gamma^{-1}$.
The main point in the proof is to reduce to the case of $G$ being an abelian variety, a torus, or a vector group, by a method similar to the one employed in \cite{BCH}. Here, we give a more ad hoc discussion (avoiding cohomology) and simplify matters using the commutativity of $G$. 

We first observe that if $N$ is a connected
 normal algebraic subgroup of $G$ stable under $\End(G)$, then $\tau$ induces an endomorphism $\tau_{G/N}$ of $G/N$. We claim that $\tau_{G/N}$ is confined  and that \begin{equation} \label{decfix} \# \mathrm{Fix}(\tau) = \# \mathrm{Fix}(\tau|_N) \cdot \# \mathrm{Fix}(\tau_{G/N}). \end{equation} 
To see this, we first note that by Lemma \ref{lem:silvertwist}\ref{state:autotwist} powers $\tau^k$ of $\tau$ are of the form $\gamma'\sigma^{nk}$ for some $\gamma'\in \Gamma$, and hence by Lemma \ref{lem:silvertwist}\ref{state:dynamicallyconfined} $\tau$ is confined. Since $N$ is connected and the map $\tau$ is confined, we get that $\tau|_N-1$ is an isogeny (in particular, it is surjective), which implies that the map $ \mathrm{Fix}(\tau)  \rightarrow \mathrm{Fix}(\tau_{G/N})$ is surjective as well. Applying this to $\gamma=1$ shows that the map $\sigma_{G/N}$ is confined, and we get an exact sequence of finite groups 
\begin{equation*} 0 \rightarrow \mathrm{Fix}(\tau|_N)  \rightarrow \mathrm{Fix}(\tau)  \rightarrow \mathrm{Fix}(\tau_{G/N}) \rightarrow 0. \end{equation*} 
Notice also that $\tau|_N = \sigma|_N^n \gamma|_N^{-1}$ and  $\tau_{G/N} = \sigma_{G/N}^n \gamma_{G/N}^{-1}$ admit the same decomposition as $\tau$ with $\sigma|_N$ (resp., $\sigma_{G/N}$) being a confined isogeny on $N$ (resp., $G/N$).

We apply (\ref{decfix}) several times: first, by Chevalley's structure theorem for algebraic groups \cite[Thm.~1.1]{ConradChev}, $G$ has a unique normal connected linear algebraic subgroup $N$ such that $G/N$ is an abelian variety.  Then $N$ is stable by $\End(G)$, since there are no nontrivial morphisms from a linear algebraic group $N$ to an abelian variety $A$ \cite[Lem.~2.3]{ConradChev}. Now suppose $G$ is a connected commutative linear algebraic group; then there exists a normal connected unipotent algebraic subgroup $U$ of $R$ such that the quotient $R/U$ is a torus $T$, i.e.\ isomorphic to $\Gm^s$ for some $s\in\Z_{\geqslant 0}$ \cite[Thm.~16.33]{MilneAG}. There are no nontrivial morphisms $U\to T$ \cite[Cor.~IV.2.2.4]{DGGA}, so $U$ is preserved by any endomorphism of $R$. Now if $G$ is connected commutative unipotent, it is  isogenous to a direct product $W_1\tm\cdots\tm W_t$ of additive groups of truncated Witt vectors \cite[Thm.~VII.1]{Serre}. Since $p^{d} W_i=0$ for some $d$, we obtain a decomposition series of $G$ (using \cite[Prop.~IV.2.2.3]{DGGA}) $G\supseteq pG\supseteq p^2G\supseteq\dots\supseteq 0,$
in which $pG$ is preserved by any endomorphism of $G$, and each successive quotient is a connected commutative unipotent algebraic group of exponent $p$. By \cite[Prop.~VII.11]{Serre}, such a group is isomorphic to a vector group $\Ga^r$ for some $r\in\Z_{\geqslant 0}$. 

By the above discussion, we are reduced to considering the following three cases. In each of these cases, $G$ is connected commutative, $\End(G)$ is a ring with a degree function $\deg \colon \End(G) \rightarrow \N\cup\{-\infty\}$ and $\# \ker (\sigma^n-\gamma) \leqslant \deg(\sigma^n-\gamma)$, so it suffices to prove that in each of these cases $\deg(\sigma^n-\gamma)$ grows at most exponentially in $n$.

\begin{itemize}[leftmargin=*,labelindent=0.4cm]
\item[$-$]{\bf $G$ is an abelian variety}:  $G$ is isogenous to a product of simple abelian varieties, and $\deg(\sigma^n-\gamma)$ becomes a product of reduced norms $N(\sigma_i^n-\gamma_i)$ on finitely many simple $\Q$-algebras $R_i$ (with $\tau_i \in R_i$ and $\gamma_i \in R_i^{\tm}$) \cite[Prop.~2.3]{D1}. Passing to the algebraic closure of $\Q$, one easily sees that these satisfy a linear recurrence in $n$, and hence grow at most exponentially. 
\item[$-$] {\bf $G \simeq \Gm^s$ is a torus}: Identifying endomorphisms of $G$ with matrices in $\mathrm{M}_s(\Z)$, one sees (e.g.\ by using the Smith normal form) that $$\deg(\sigma^n-\gamma) = |\det(\sigma^n-\gamma) |. $$ Expanding the determinant shows the desired growth behaviour. 
\item[$-$] {\bf $G \simeq \Ga^r$ is a vector group}: Endomorphisms of $G$ are given by $r \tm r$ matrices over the skew polynomial ring $K\langle \phi \rangle$ with $\phi a = a^p \phi$ for $a \in K$. The degree of an isogeny $\tau \in \End(G)$ can be computed using the Dieudonn\'e determinant $\mathrm{ddet}$ 
by the formula $$\deg(\tau)=p^{\deg_{\phi} \mathrm{ddet}(\tau)}.$$ 
(Since $K\langle\phi\rangle$ is left and right euclidean, we can put the matrix $\tau$ in Smith normal form and use the fact that unimodular matrices have Dieudonn\'e determinant of degree $0$ \cite[Thm.~4.6]{GiesbrechtKim}). 
We will use that if $\tau$ is a matrix over $K\langle \phi \rangle$ all of whose entries have degree $\leqslant d$ as polynomials in $\phi$, then $\deg_{\phi} \mathrm{ddet}(\tau) \leqslant rd$  \cite[Thm.~3.5]{GiesbrechtKim}. 
Choose an integer $d\geqslant 1$ so that the entries of $\sigma$ have degree $\leqslant d$. For sufficiently large $n$ the entries of $\sigma^n-\gamma$ have degree $\leqslant nd$, and hence \[\deg(\sigma^n-\gamma) = p^{\deg_{\phi} \mathrm{ddet}(\sigma^n-\gamma)}\leqslant p^{nrd}. \qedhere \]
\end{itemize} \end{proof}

\begin{remark} For a more comprehensive treatment of degrees and inseparable degrees of endomorphisms of algebraic groups (not necessarily commutative), we refer the reader to \cite{BCH}.
\end{remark}

In the above proof, the positive radius of convergence of $\zeta_f(z)$ is recursively defined based on a decomposition of $G$ along subgroups and quotients. The following is a case where we can find a direct bound on the radius of convergence: 

\begin{proposition} \label{smoothproj} When $V$ is \emph{smooth projective} and $f \colon V \rightarrow V$ is any morphism, $\zeta_f(z)$ and $\zeta^*_f(z)$ define holomorphic functions on 
a disk of radius the smallest absolute value of a zero of the characteristic polynomial $$\det (1-f^*z \mid \mathrm{H}^{2 \bullet}(V)),$$ where $\mathrm{H}^{2\bullet}(V) = {\displaystyle{\bigoplus\limits_{j=0}^{\dim V}}} \mathrm{H}^{2j}(V,\Q_\ell)$ is the even \'etale cohomology of $V$ for some $\ell \neq p$. 
\end{proposition} 

\begin{proof} 
In this case, we have a coefficient-wise bound $f_n \leqslant (\Gamma_{f^n} \cdot \Delta)$, where the right hand side is the intersection number of the graph of $f^n$ with the diagonal in $V \times V$. Since $V$ is smooth projective and $f$ has finitely many fixed points, by the Grothendieck--Lefschetz fixed point formula in $\ell$-adic cohomology for $\ell \neq p$ \cite[Cor.~3.7, p.\ 152 (= Expos\'e ``Cycle'', p.\ 24)]{SGA412} we  find that \begin{equation} \label{GLfpf} \exp \left( \sum (\Gamma_{f^n} \cdot \Delta) z^n/n \right) = \prod_{i=0}^{2 \dim V} \det (1-f^*z \mid \mathrm{H}^i(V,\Q_\ell))^{(-1)^{i+1}}, \end{equation}
and the right hand side converges in an open disk of radius the smallest absolute value of a zero of the denominator. 
\end{proof} 

\begin{remark} If $V=\PP^k$ is a projective space, the result follows essentially from B\'ezout's theorem (see e.g.\ \cite[Prop.~1.3]{Sibony}). A more general ``intersection-theoretic'' argument such as in the proof of Proposition \ref{smoothproj} seems to apply only in a restrictive setting, since the Grothendieck--Lefschetz fixed point formula can fail for general endomorphisms of general varieties, and one cannot in general leave out the assumptions of properness and smoothness. It seems these assumptions are rarely satisfied, as witnessed by the following sample result in characteristic zero: if $\Gamma$ is a finite group of endomorphisms of a complex abelian variety of dimension $\geqslant 3$ acting irreducibly on the tangent space at $0$, $\Gamma \backslash G$ is necessarily a projective space, $G$ is a power of an elliptic curve, and $\Gamma$ is one of two possible groups (depending on the dimension) \cite[Thm.\ 1.1]{ALA}. 
\end{remark} 

\section{\mbox{ } \\ Explicit computation of tame zeta functions\\ for some dynamically affine maps on \texorpdfstring{$\PP^1$}{}} 
\label{lois}

\subsection*{Classification of dynamically affine maps on $\PP^1$} Bridy \cite{BridyBordeaux} classified all dynamically affine maps $f$ of degree $\geqslant 2$ on the projective line by showing that they are conjugate by a fractional linear transformation to  polynomials $f_c$ in an explicit standard form, as in Table \ref{tab:da} (where $\mu_d\subseteq k^{\tm}$ denotes a nontrivial subgroup consisting of $d$-th roots of unity). This is the characteristic-$p$ analogue of the classification over $\C$ given in \cite[Thm.~3.1]{MilnorBodilpaper}. \\[-1cm]

\begin{table}[ht] 
\centering
\def\arraystretch{1.5}
\begin{tabular}{|c|c|c|c|c|}
\multicolumn{4}{c}{}\\
\hline
$G$ & $\G$ & $\Gamma\backslash G$ & $f_c$ \\
\hline
\multirow{2}{*}{$\Ga$} & $\{1\}$ & \multirow{2}{*}{${\PP^1}- \{\infty\}$} & Additive polynomial \\
& $\mu_d$ & & Subadditive polynomial \\
\hline
\multirow{2}{*}{$\Gm$} & $\{1\}$ & ${\PP^1} - \{0,\infty\}$ & Power map \\
& $\{\pm 1\}$ & ${\PP^1} - \{\infty\}$ & Chebyshev map \\
\hline
$E$ & $\neq\{1\}$ & $\PP^1$ & Latt\`es map\\
\hline
\end{tabular}
\caption{Classification of dynamically affine maps on $\PP^1$.}
\label{tab:da} 
\end{table}

\mbox{ } \\[-15mm]

\noindent With the notation and terminology of Table \ref{tab:da}, $f$ is coseparable precisely when either $f_c$ is inseparable, or $f_c$ is a separable (sub)additive polynomial for which $f_c'(0)$ is transcendental over $k$ (cf.\ \cite[Thm.~1.2 \& 1.3]{BridyBordeaux}). One easily checks that these are precisely the maps for which $f_n$ is \emph{maximal} for all $n$ (i.e.\ $f_n = \deg(f)^n+1$; each fixed point of $f^n$ has multiplicity one).

\subsection*{Some examples of tame zeta functions} The ``trivial'' case provides us with a useful notational tool:  if $X=\mathrm{pt}$ is a point, then $f$ has a unique fixed point ($f_n=1$ for all $n$), so we suppress the (irrelevant) $f$ from the notation, to obtain $$\zeta_{\mathrm{pt}}(z)=\frac{1}{1-z} \qquad \mbox{ and }\qquad  \zeta^*_{\mathrm{pt}}(z) = \zeta_{\mathrm{pt}}(z)/\sqrt[p]{\zeta_{\mathrm{pt}}(z^p)} = \frac{\sqrt[p]{1-z^p}}{1-z}. $$
We will now present examples of tame zeta functions, writing them in a concise form using the function $\zeta^*_{\mathrm{pt}}(az^b)$ for various $a$ and $b$. 
Much of the general structure of (tame) zeta functions of algebraic groups is already visible in the following basic example for which we provide a detailed computation (we stick to $p>2$ for convenience).

\begin{proposition} \label{zm}  Let $m\geqslant 2$ be an integer and let $f \colon \PP^1 \rightarrow \PP^1, x \mapsto x^m$ be the power map over an algebraically closed field $K$ of characteristic $p>2$. If $p$ divides $m$, set $\beta:=0$. Otherwise, let $\beta:= (|m^s-1|_p-1)/s \in \Z[1/p]$, where $s$ is the smallest positive  integer for which $|m^s-1|_p<1$ \textup{(}i.e.\ the order of $m$ in $\F_p^{\tm}$\textup{)}. Then 
$$ \zeta^*_{f}(z)  
= {\ZP(mz)}{\ZP(z)} \left( \frac{\ZP((mz)^s)}{\ZP(z^s)} \right)^\beta. $$
\end{proposition}

\begin{proof} 
The iterate $f^{ n}$ has as its fixed points $\infty, 0$, and the distinct solutions to $x^{m^n-1}=1$ in $\overline \F_p$. 
Hence $f_n = 2+(m^n-1) \cdot |m^n-1|_p.$ 

If $p{\mid}m$, we have $f_n=m^n+1$, and the result follows. Now assume that $m$ is coprime to $p$. We then have (see e.g.\  \cite[Lemma 6.1]{BridyBordeaux})  \begin{equation} \label{fnex} |m^n-1|_p = \left\{ \begin{array}{ll} 1 & \mbox{ if } s{\nmid}n, \\ |m^s-1|_p \cdot |n|_p & \mbox{ if } s{\mid}n. \end{array} \right. \end{equation} 
Observe that $s$ is a divisor of $p-1$; in particular, $s$ is coprime to $p$ and $\beta \in \Z[1/p]$. If we set $M:=|m^s-1|_p$, the tame zeta function can be computed as follows:
\begin{align*} 
& \log\, \zeta^*_{f}(z)/\ZP(z)^2 =  \sum_{ p \nmid n; s \nmid n} \frac{m^n-1}{n}z^n +  M \sum_{ p \nmid n; s \mid n} \frac{(m^n-1)}{n} z^n  \\ 
& = \sum_{n} \frac{m^n-1}{n}z^n  - \sum_{n} \frac{m^{pn}-1}{pn} z^{pn} \\ & \hspace*{75pt} + (M-1)\left(\sum_{ n}  \frac{m^{sn}-1}{sn} z^{sn} - \sum_ { n}  \frac{m^{psn}-1}{psn} z^{psn}\right)\\ 
& =\log\left( \frac{1-z}{1-mz} \cdot \frac{(1-(mz)^p)^{\frac{1}{p}}}{(1-z^p)^{\frac{1}{p}}}\right) + \frac{(M-1)}{s}\log\left( \frac{1-z^s}{1-(mz)^s} \cdot \frac{(1-(mz)^{ps})^{\frac{1}{p}}}{(1-z^{ps})^{\frac{1}{p}}}\right). \qedhere
\end{align*}
\end{proof}

\noindent Without showing further details of computations (that go along the lines of those for the power map) we now list several other tame zeta functions.  
\begin{proposition} Suppose that $K$ is an algebraically closed field of characteristic $p>2$. 
Let $m\geqslant 2$ be an integer. The normalised Chebyshev polynomial $T_m$ is the unique monic polynomial of degree $m$ with integer coefficients satisfying $T_m(z+z^{-1}) = z^m+z^{-m}$. Consider the \emph{Chebyshev map}  $$T_m \colon \PP^1 \rightarrow \PP^1$$ given by the polynomial $T_m$ \textup{(}denoted by the same symbol\textup{)}. 

Let $E/K$ denote an elliptic curve and let $\pi \colon E \rightarrow \PP^1$ be a covering map of order two. The \textup{(}standard\textup{)} \emph{Latt\`es map} $$L_m \colon \PP^1 \rightarrow \PP^1$$  corresponding to $\pi$ is defined by the property $L_m \circ \pi = \pi \circ [m]$, where $[m]$ is the multiplication-by-$m$ map on $E$. 

If $p{\nmid}m$, let $s$ denote the multiplicative order of $m$ modulo $p$. Let $h=1$ if $f$ is a Chebyshev map or a Latt\`es map arising from an ordinary elliptic curve, and let $h=2$ otherwise. Set  $\beta:=(|m^s-1|^h_p-1)/s \in \Z[1/p]$. 
Then the corresponding tame zeta functions \textup{(}quotiented by a convenient factor\textup{)} are given in Table \ref{tab:da2}.

\begin{table}[h] 
\centering
\def\arraystretch{1.5}
\begin{tabular}{| L | L | L |}
\hline 
\mbox{\textup{Condition}} &  \zeta^{*}_{T_m}(z)/ (\ZP(mz) \ZP(z))  & \zeta^{*}_{L_m}(z)/(\ZP(m^2z) \ZP(z))\\[3pt]
\hline 
p {\mid} m & 1 & 1  \\[5pt]
p {\nmid}m \mbox{ \textup{and} } 2 {\nmid} s & \left(\frac{\ZP((mz)^s)}{\ZP(z^s)}\right)^{\beta/2} & \left( \frac{\ZP((m^2z)^s) \ZP(z^s)}{\ZP((mz)^s)^2} \right)^{\beta/2}  \\[20pt]
p {\nmid}m \mbox{ \textup{and} } s=2t & \left(\frac{\ZP((mz)^t)\ZP(z^t)}{\ZP(z^{2t})}\right)^{\beta} &
\left( \frac{\ZP((m^2z)^t) \ZP(z^t) \ZP((mz)^t)^2}{\ZP((mz)^{2t})^2} \right)^\beta \\[15pt]
\hline
\end{tabular}
\caption{Tame zeta functions of some dynamically affine maps on $\PP^1$.}
\label{tab:da2} 
\end{table}
\vspace*{-25pt} \qed
\end{proposition} 

\begin{proposition} Suppose that $K$ is an algebraically closed field of characteristic $p>0$ and consider an \emph{additive polynomial} in $K[X]$ of the form $a_0 X + a_1 X^p + \cdots +a_m X^m$ with $m=p^r$ for some integer $r\geqslant 1$. Assume that $a_0 \in \overline{\F}_p^{\tm}$ and $m\geqslant 2$. Consider $f$ as a map $$ f \colon \PP^1 \rightarrow \PP^1, X \mapsto  a_0 X + a_1 X^{p} + \dots + a_m X^{m}.$$ Let $s\geqslant 1$ be the smallest integer with $f^{s}(X) = X + a X^{p^t} + \cdots$ for $a \neq 0$ and $t\in \Z_{>0}$. Put $\beta=(p^{-t}-1)/s$. Then 
\[  \pushQED{\qed} 
\zeta_{f}^*(z) = \ZP(m z) \ZP(z) \ZP((m z)^s)^\beta. \qedhere \popQED \]
\end{proposition}

\noindent In characteristic two and for more general (sub)additive polynomials, similar methods apply, but the computations are more tedious and we have not listed the outcome. We have not carried out an explicit computation for general Latt\`es maps arising from endomorphisms of elliptic curves that are not given by multiplication by an integer or corresponding to  larger (possibly noncommutative) automorphism groups $\Gamma$.

\bibliographystyle{amsplain}
\providecommand{\bysame}{\leavevmode\hbox to3em{\hrulefill}\thinspace}
\providecommand{\MR}{\relax\ifhmode\unskip\space\fi MR }
\providecommand{\MRhref}[2]{%
  \href{http://www.ams.org/mathscinet-getitem?mr=#1}{#2}
}
\providecommand{\href}[2]{#2}

\end{document}